\pdfoutput=1
\RequirePackage{ifpdf}
\ifpdf 
\documentclass[pdftex]{sigma}
\else
\documentclass{sigma}
\fi

\numberwithin{equation}{section}

\newtheorem{Theorem}{Theorem}[section]
\newtheorem{Corollary}[Theorem]{Corollary}
\newtheorem{Lemma}[Theorem]{Lemma}
\newtheorem{Proposition}[Theorem]{Proposition}

\theoremstyle{definition}
\newtheorem{Definition}[Theorem]{Definition}

\newtheorem{Example}[Theorem]{Example}
\newtheorem{Remark}[Theorem]{Remark}
\newtheorem{notation}[Theorem]{Notation}

\usepackage{tikz-cd}
\usepackage{enumerate}
\usepackage{enumitem}

\newcommand{\BBR}{\mathbb{R}}
\newcommand{\BBC}{\mathbb{C}}
\newcommand{\BBZ}{\mathbb{Z}}

\newcommand{\BBN}{\mathbb{N}}
\newcommand{\BBS}{\mathbb{S}}
\newcommand{\CG}{\mathcal{G}}
\newcommand{\CC}{\mathcal{C}}
\newcommand{\CH}{\mathcal{H}}
\newcommand{\CK}{\mathcal{K}}

\renewcommand{\CD}{\mathcal{D}}

\newcommand{\stardot}{\overset{{\bullet}}{*}}
\newcommand{\starcirc}{\overset{\circ}{*}}
\newcommand{\inv}{^{-1}}
\newcommand{\darrow}{\arrow[d, shift right] \arrow[d, shift left]}
\newcommand{\rarrow}{\arrow[r, shift right] \arrow[r, shift left]}
\newcommand{\grpd}{\rightrightarrows}
\DeclareMathOperator{\pr}{pr}

\DeclareMathOperator{\Mat}{Mat}
\DeclareMathOperator{\SU}{SU}

\DeclareMathOperator{\GL}{GL}

\begin{document}

\allowdisplaybreaks

\newcommand{\arXivNumber}{2312.00341}

\renewcommand{\PaperNumber}{093}

\FirstPageHeading

\ShortArticleName{Convolution Algebras of Double Groupoids and Strict 2-Groups}

\ArticleName{Convolution Algebras of Double Groupoids\\ and Strict 2-Groups}

\Author{Angel ROM\'AN~$^{\rm a}$ and Joel VILLATORO~$^{\rm b}$}

\AuthorNameForHeading{A.~Rom\'an and J.~Villatoro}

\Address{$^{\rm a)}$~Washington University in St.~Louis, USA}
\EmailD{\href{mailto:angelr@wustl.edu}{angelr@wustl.edu}}

\Address{$^{\rm b)}$~Indiana University Bloomington, USA}
\EmailD{\href{mailto:joeldavidvillatoro@gmail.com}{joeldavidvillatoro@gmail.com}}

\ArticleDates{Received February 12, 2024, in final form October 09, 2024; Published online October 19, 2024}

\Abstract{Double groupoids are a type of higher groupoid structure that can arise when one has two distinct groupoid products on the same set of arrows. A particularly important example of such structures is the irrational torus and, more generally, strict 2-groups. Groupoid structures give rise to convolution operations on the space of arrows. Therefore, a double groupoid comes equipped with two product operations on the space of functions. In this article we investigate in what sense these two convolution operations are compatible. We use the representation theory of compact Lie groups to get insight into a certain class of 2-groups.}


\Classification{46L05; 58B34; 46L87; 18G45; 18N10; 58H05}

\section{Introduction}
Given a Lie groupoid $ \CG \grpd M $ and a Haar system on $ \CG $ one can associate a $ C^*$-algebra $ C^* (\CG) $. This relationship is the fundamental link between Lie groupoids and noncommutative geometry.
The algebra $ C^*(\CG) $ can, in some sense, be thought of as the (noncommutative) algebra of functions on the differentiable stack $ [M/\CG] $.
This point of view is justified by a theorem of Muhly, Renault, and Williams~\cite[Theorem 2.8]{MuhlyRenaultWilliams} which says that if $ \CG $ and $ \CH $ are Morita equivalent groupoids (i.e., they represent the same stack), then the associated $ C^* $ algebras $ C^* (\CG ) $ and $ C^*(\CH)$ are strongly Morita equivalent.
In particular, if $ \CG $ is Morita equivalent to a manifold then $ C^*(\CG) $ is strongly Morita equivalent to the algebra of functions on the smooth quotient space $ M / \CG $.

The objective of this article is to try to better understand $C^* $ algebras of some higher structures (namely double groupoids). In principle, these algebras should be a model for a type of ``noncommutative groupoid''.

To better explain what we mean by this, let us consider our model example: The noncommutative torus is a noncommutative space that arises from the convolution algebra of the action groupoid $ \BBZ \ltimes S^1 $ where $ \BBZ $ acts on $ S^1 $ by irrational rotations. The groupoid composition gives rise to a convolution product on the vector space of compactly supported functions $ C_c^\infty\big(\BBZ \ltimes S^1\big) $. As we mentioned earlier, we think of this algebra as a model for the algebra of ``functions'' on the space $ S^1/\BBZ $.
However, we should note that $ S^1/\BBZ$, while not smooth, is a perfectly well-defined group.
Since $ C_c^\infty\big(\BBZ \ltimes S^1\big) $ represents the functions on this singular quotient, we should, in principle, expect it to inherit two algebra structures. One of the algebra structures should be analogous to the product arising from pointwise multiplication while the other should be analogous to the \emph{group} convolution product.

Now observe that $ \BBZ \ltimes S^1 $ actually has two natural groupoid structures. One, denoted $ \circ$, arises from treating $ \BBZ \ltimes S^1 $ as an action groupoid while the other, denoted $ \bullet $, arises from treating $ \BBZ \ltimes S^1 $ as a Cartesian product of groups. These two product operations are compatible in the sense that
\begin{equation}\label{eqn:intro.compat}
(a \circ b) \bullet ( c \circ d) = (a \bullet c) \circ (b \bullet d )
\end{equation}
for all $ a,b,c,d \in \BBZ \ltimes S^1 $, suitably composable. The operations $ \circ $ and $ \bullet $ give rise to two different convolution operations on $ C_c^\infty\big(\BBZ \ltimes S^1\big) $ which we denote \smash{$ \starcirc $} and \smash{$ \stardot $}. One of the main aims of this article is to consider the question ``In what sense are \smash{$ \stardot $} and \smash{$ \starcirc$} compatible?''

The general context for this question is that of a double groupoid. A double groupoid is a~groupoid object in the category of groupoids. From the point of view of stacks, such an object can be thought of as a (strict) groupoid in the category of stacks. From the noncommutative geometry point of view, the convolution algebras of a double groupoid should be a type of ``noncommutative groupoid''.

The primary feature of a double groupoid $ \CG $ is that it has two product operations $ \bullet $ and $ \circ $ and these two product operations are \emph{compatible} in the sense that they satisfy equation~\eqref{eqn:intro.compat} whenever both sides of the equation are well defined.
If we choose Haar systems, then we get two convolution operations \smash{$ \starcirc $} and \smash{$ \stardot $} on $ C_c^\infty(\CG) $. Now let us break the symmetry of the situation by considering \smash{$ \starcirc$} to be the algebra of functions on the noncommutative space defined by $ \circ $. Then, intuitively, \smash{$ \stardot$} should correspond to ``convolution''.

When discussing the compatibility of \smash{$ \starcirc $} and \smash{$ \stardot $}, the naive guess would be to assume that they satisfy a version of equation~\eqref{eqn:intro.compat}. However, the situation is not so simple. There is well-known lemma of Eckmann and Hilton which provides us with some hints as to why one should be careful.
\begin{Lemma}[Eckmann--Hilton~\cite{EckmannHilton}]
   Suppose a set $ A $ is equipped with globally defined, unital, binary operations $ \circ $ and $ \bullet $ and assume that these two products satisfy equation~\eqref{eqn:intro.compat}. Then $ \circ $ is associative, commutative, and $ \circ = \bullet $.
\end{Lemma}
At first glance, this lemma seems to suggest that there are no interesting double groupoids. However, the key point is that the binary operations in the Eckmann--Hilton lemma must be globally defined. Furthermore, a problem arises once we pass to the convolution algebras \smash{$ \stardot$} and~\smash{$ \starcirc$} which are both globally defined on $ C_c^\infty(\CG) $. Hence, the lemma of Eckmann and Hilton suggests that the naive notion of compatibility between \smash{$ \starcirc$} and \smash{$ \stardot$} does not hold.\footnote{Since convolution algebras are not typically unital, we cannot directly apply the Eckmann--Hilton lemma here. However, convolution algebras are not that far from being unital and one could reasonably expect that the argument can be adapted to this setting.}

There two main aims for this paper. One is to establish some basic facts and definitions about the general case. However, as we will see, a comprehensive study of the compatibility properties for the general case is likely very complicated. After discussing the general case, we will study three main examples. The goal with our examples is to find formulas resembling equation~\eqref{eqn:intro.compat} that \emph{are} true for the two convolution products. Our three main examples are (1)~the irrational torus, (2)~compact Lie groups (thought of as a trivial double groupoid) and (3)~``compact singular Lie groups''. Our strategy for all three of these cases will be to take a~fairly computation focused approach that explores the structure of the two convolution products.

Let us briefly summarize the main results for the last case, since it is the most general. Suppose $ \CG = K \ltimes G $ where $ K $ is a (possibly non-closed) discrete normal subgroup of a compact, connected, Lie group $ G $. We let $ \circ $ denote the composition operation on $ \CG $ from the action groupoid structure and let $ \bullet $ denote the group product. Such a $ \CG $ is what we call a \emph{compact singular Lie group}. Now let \smash{$ \starcirc $} and \smash{$ \stardot $} be the associated convolution operations on $ C_c^\infty(\CG) $.

Given a smooth function $ u \in C^\infty(G) $ and $ \kappa \in K $ write $ u^\kappa \in C^\infty_c(K \times G) $ to denote the function
\[
u^\kappa (x,y) :=
\begin{cases}
  u(y)  & \text{if } x = \kappa, \\
  0  & \text{otherwise}.
\end{cases}
\]
Now suppose $ \pi $ is an $ n $-dimensional representation of $ G $. By choosing an ordered basis for the underlying vector space and taking $ 0 \le i,j \le n $, we can recover a function $ \pi_{ij} \in C^\infty(G, \BBC) $ by taking a matrix coefficient of the representation.

We now state the main theorem.

\begin{Theorem}
Let $ \pi $ be an $ n$-dimensional representation of $ G $. Suppose that we have $u,v \in C^\infty(G)$, $ \kappa_1, \kappa_2, \lambda_1, \lambda_2 \in K $ and
$ 1 \le i, k \le n$.
Then we have that
  \begin{equation}\label{eqn:main.1}
   \sum_{j = 1}^{n} \big( u^{\kappa_1} \starcirc \pi_{ij}^{\lambda_1} \big) \stardot \big( v^{\kappa_2} \starcirc \pi_{jk}^{\lambda_2} \big) =   \sum_{j = 1}^{n} \big( u^{\kappa_1} \stardot  v^{\kappa_2} \big) \starcirc \big(\pi_{ij}^{\lambda_1} \stardot \pi_{jk}^{\lambda_2} \big).
  \end{equation}
\end{Theorem}
If the above theorem did not have sums, it would describe a set of 4-tuples where a compatibility condition of the form \eqref{eqn:intro.compat} holds. The stronger version of equation~\eqref{eqn:main.1} (i.e., without the sums) is false in general. However, versions of it hold for some of the more tame examples such as the case where $ G $ is a torus.

An interpretation of the above formula is to say that the failure of the compatibility law for vanishes ``on average'' for certain expressions involving matrix coefficients.

\subsection{Additional remarks}
To the best of our knowledge, there is not much by way of existing literature on convolution algebras for double groupoids. Sections \ref{sec2}--\ref{sec5} are largely dedicated to ensuring that this article can be a relatively self-contained introduction to the topic. Notably, in Section \ref{sec5} we introduce the notion of a double Haar system which, to our knowledge, is new. A definition of a $ 2$-Haar system appears in Amini~\cite{Amini} but there does not seem to be any compatibility condition and so it seems to be too weak to be useful for our purposes.

Sections \ref{Section:noncommutative.torus}--\ref{Section:compact.singular.Lie.groups} are dedicated to using some representation theory of compact Lie groups to compute some interesting examples. These calculations we do here lean on the fact that the representation theory of compact Lie groups is closely related to the convolution operation. This allows us to take advantage of some useful properties of matrix coefficients. In the future it could be interesting to better understand the case of non-compact Lie groups where the representation theory can be significantly more subtle.

The compatibility law (equation~\eqref{eqn:main.1}) that we have proved relies on specific facts about the structure of the convolution algebras involved. We suspect that there is a more general form which might lead to a notion of ``compatible algebras''. However, at this time, the correct notion of compatible algebras is not completely clear.

There are also alternative approaches one could take to investigate this topic. Rather than considering two algebra structures on the same set, one could instead attempt to study these structures in the form of bialgebras or Hopf algebras. Some work in this direction does exist (see, for example, Hopfish algebras~\cite{Hopfish2,Hopfish}).

\section{Groupoids background}\label{sec2}
In this section, we will review some basic facts about and establish our notation for Lie groupoids and related algebraic structures.
\subsection{Categories}
The most general kind of algebraic object we will examine is that of a (small) category. Let us give a definition so that we can establish some of our notation conventions and terminology for categories. Our conventions will reflect the fact that we will be considering categories as algebraic objects in the vein of a group or a monoid.
\begin{Definition}\label{def:category}
  A \emph{$($small$)$ category} $ \CC = \CC_1 \grpd \CC_0 $ is a pair of sets $\CC_1 $ (the \emph{arrows}) and $\CC_0 $ (the \emph{objects}) together with a pair of functions
  $s \colon \CC_1 \to \CC_0$, $t \colon \CC_1 \to \CC_0$
  called the \emph{source} and \emph{target}, and another function:
  a function
  $u \colon \CC_0 \to \CC_1$, $x \mapsto 1_x$
  called the \emph{unit}. The set of \emph{composable arrows} is defined to be
  \[ \CC_1 \times_{s,t} \CC_1 := \{ (a,b) \in \CC_1 \mid s(a) = t(b) \} \]
  and we have a function
  $m \colon \CC_1 \times_{s,t} \CC_1 \to \CC_1$, $(a,b) \mapsto a \circ b$
  called \emph{multiplication} or \emph{composition}. We require that these functions satisfy the following axioms:
  \begin{itemize}\itemsep=0pt
    \item (Compatibility of source and target with multiplication):
    $\forall (a,b) \in \CC_1 \times_{s,t} \CC_1$, $ t( a \circ b) = t(a) $ and $ s(a \circ b ) = s(b)$.
    \item (Associativity):
    $\forall (a,b,c ) \in \CC \times_{s,t} \CC \times_{s,t} \CC$, $ (a \circ b) \circ c = a \circ (b \circ c)$.
    \item (Compatibility of the unit with source and target):
    $\forall x \in \CC_0$, $ s(1_x) = t(1_x)$.
    \item (Left and right unit laws):
   $\forall a \in \CC_1$, $ 1_{t(a)} \circ a = a  = a \circ 1_{s(a)}$.
  \end{itemize}
   If the set of objects $ \CC_0 $ is a singleton, then we say that $ \CC $ is a \emph{monoid}. If we remove the unit map and its associated axioms the resulting structure is called a \emph{semi-category} and a semi-category with only one object is called a \emph{semi-monoid} (see Definition~\ref{def:category}).
\end{Definition}
In general, we will use the notation $ A \grpd B $ to indicate that $ A$ is the arrows of a category with objects $ B $.
Let us now consider a few basic examples of categories that will play a role in our discussion.
\begin{Example}
   Let $ \CC_0 = \{ * \} $ and let $ \CC_1 = \BBN $. If we take composition to be addition of natural numbers, then this constitutes a category. Since it only has one object, it is an example of a~monoid.
\end{Example}
\begin{Example}[complex Euclidean representations]
  Let $ G $ be a Lie group and let
  \[ \CC_0 = \{ \rho \colon G \to \GL(n,\BBC) \mid n \in \BBN,\, \rho \text{ representation} \} \]
   be the set of all representations of $ G $ on complex Euclidean spaces. Given an object $ \rho \in \CC_0 $ let us write $ V_\rho $ to denote the underlying complex Euclidean space. Now let
  \[ \CC_1 := \{ (\rho_2, L, \rho_1) \mid \rho_1, \rho_2 \in \CC_0, \, L \colon V_{\rho_1} \to V_{\rho_2} \text{ linear intertwining operator} \}. \]
  Here, linear intertwining operator means that $ L $ is a linear map and for all $ g \in G $, $ v \in V_{\rho_1} $ we have that
  $ L(\rho_1(g) v) = \rho_2(g) L(v)$.
  It is not too difficult to see that this is an example of a~category if we take the multiplication operation to be
 $(\rho_3, L_2, \rho_2) \circ (\rho_2, L_1, \rho_1 ) = (\rho_3, L_2 L_1, \rho_1 )$.
\end{Example}

\subsection{Groupoids}
Briefly, a groupoid is a category $ \CG $ where every arrow is invertible. In our notation conventions, we will typically use $ \CC $ to denote a category and $ \CG $ to denote a groupoid.
\begin{Definition}
A \emph{groupoid} $ \CG $ is a category equipped with a function
$i \colon \CG_1 \to \CG_1$, $g \mapsto g\inv$
which satisfies the following properties:
\begin{itemize}\itemsep=0pt
  \item (Compatibility of inverse with source and target):
  $\forall g \in \CG_1$, $ s\big(g\inv\big) = t(g)$ and \linebreak $t\big(g\inv\big) = s(g)$.
  \item (Inverse law):
  $\forall g \in \CG_1$, $ g\inv \circ g = u(s(g))$ and $g \circ g\inv = u(t(g))$.
\end{itemize}
\end{Definition}
\subsection{Examples of groups and groupoids}
In this section, we will lay out some of the most basic examples of groups and groupoids. Of particular importance to us will be the action groupoid. Action groupoids can be used to construct some of the most basic non-trivial examples of double structures.
\begin{Example}[groups]
  Suppose $ G $ is a group.
  Let $ \{ * \} $ be a set with a single point and take $ s \colon G \to \{ * \} $ and $ t \colon G \to \{ * \} $ to be the unique functions.
  Take $ u \colon \{ * \} \to G$ to be the constant map which sends $ * $ to the neutral element.
  If we take $ i $ and $ m$ to be the usual inverse and multiplication maps for the group, then we get the structure of a groupoid $ G \grpd \{ * \} $.
\end{Example}
\begin{Example}[group actions]
Suppose $ G $ is a group acting on a set $ X $.
Let $ e \in G $ denote the neutral element and denote the action by $ (g,x) \mapsto g \cdot x $.

We can construct a groupoid structure $G \times X \grpd X $.
The source and target maps are as follows~${s(g,x) = x}$, $t(g,x) = g \cdot x$.
The unit and inverse maps are $u(x) = (e,x)$, ${i(g,x) = \big(g\inv, g \cdot x \big)}$.
Finally, the multiplication map is $(g_1, g_2 \cdot x) \circ (g_2, x) = (g_1 \cdot g_2, x )$.
\end{Example}

\begin{Example}[equivalence relations]
  Suppose $ E \subset X \times X $ is an equivalence relation on a~set~$ X$.
  Then we can form a groupoid $ E \grpd X $.
   The source and target maps are $s(x,y) = y$, $t(x,y) = x$, the unit and inverse maps are $u(x) = (x,x)$, $i(x,y) = (y,x)$ and the multiplication map is~${(x,y) \circ (y,z) = (x,z)}$. In the case that $ E = X \times X $, then this is called the \emph{pair groupoid}.\looseness=1
\end{Example}
\subsection{Lie groupoids}
Lie groupoids are just groupoids equipped with smooth structure. The literature on Lie group\-oids is fairly rich and we will only cover a few of the most basic concepts. For a more thorough reference, we refer to Crainic and Fernandes~\cite{LectOnInt} or Mackenzie~\cite{MackenzieGeneralTheory}.
\begin{Definition}
  A \emph{Lie groupoid} is a groupoid $ \CG \grpd M $ where the sets $ \CG $ and $ M $ are equipped with second countable smooth manifold structures.
  We further require that the source and target maps are submersions and the unit, multiplication, and inverse maps are smooth and that $ M $ is Hausdorff.
\end{Definition}
Lie groupoids have some particular features that are worth pointing out. One of them is that for each point $ x \in M $ in the object manifold, the associated source fiber $ s\inv(x) \subset \CG$ is an embedded submanifold. Even though we do not assume that $ \CG $ is Hausdorff, such source fibers of $ \CG $ are automatically Hausdorff.

Many of the examples of groupoids we provided earlier can be made into examples of Lie groupoids. For example, if a Lie group act smoothly on a manifold, then the associated action groupoid is a Lie groupoid. Groupoids associated to equivalence relations are Lie groupoids as long as the equivalence relation $ E \subset M \times M $ is an immersed submanifold and it is transverse to the fibers of each of the projection maps $ \pr_1, \pr_2 \colon M \times M \to M $. Any countable groupoid $ \CG $ can be regarded as Lie groupoid under the discrete topology.

\section{Double structures}
For our purposes, a double structure occurs when a set is equipped with more than one way to multiply elements. There are a few different contexts where this can occur and we will look at a few interesting examples.

\subsection{Compatible operations}
Let us consider the most basic type of double structure.
\begin{Definition}
  Suppose $ S $ is a set. Let $ M_\bullet $ and $ M_\circ $ be subsets of $ S \times S $ and suppose we have two binary operations
$S \times S \supset M_{\bullet} \to S$, $(a,b) \mapsto a \bullet b$, $S \times S \supset M_{\circ} \to S$, $(a,b) \mapsto a \circ b$. We say that these binary operations are \emph{compatible} if they satisfy the following equation whenever both sides of the equation are well defined:
  \begin{equation*}\label{eqn:The.Intertwining.Law}
  ( a \circ b) \bullet ( c \circ d) = (a \bullet c) \circ (b \bullet d) \qquad \text{ (compatibility law)}.
\end{equation*}
\end{Definition}
As we mentioned in the introduction, the Eckmann--Hilton lemma tells us that if a pair of compatible binary operations are unital and globally defined, then they are commutative, associative, and equal.

\begin{Example}[trivial compatible pair]
  Suppose $ \circ $ is an associative and commutative binary operation. Then $ \circ $ with itself constitutes a compatible pair.
\end{Example}
\begin{Example}[matrix multiplication and tensor product]\label{example:mat.mult.tensor}
  Let
$ \Mat = \bigsqcup_{n,m \in \BBN } M_{n \times m}(\BBR)$
  be the set of all real matrices of arbitrary dimensions.

   Let $ \bullet $ be the binary operation arising from matrix multiplication. Note that this binary operation is not globally defined since it requires that the dimensions of the matrices line up appropriately. In fact, this is the composition operation in a category where $ \Mat $ is the arrows and the objects are natural numbers. The source and target maps are just the dimension maps.

  Now, let $ \otimes $ be the tensor product (i.e., the Kronecker product of matrices)
  \[
  A \in M_{n\times m}, \  B \in M_{p \times q}, \qquad (A \otimes B )_{(i \otimes x)( j \otimes y)} := A_{ij} B_{xy},
  \]
  where we define
  $i \otimes x := (i-1)p + x$, $ j \otimes y := (j-1)q + y$.

   The operation $ \otimes $ is globally defined for all matrices and it constitutes a compatible pair with~$ \bullet $. Since $ \bullet$ is not globally defined, the Eckmann--Hilton lemma does not apply.
\end{Example}

\subsection{Double categories}
Before we can state the definition of a double category, let us make clear the definition of a~\emph{morphism of categories}, otherwise known as a functor.
\begin{Definition}
   Suppose $ \CC $ and $ \CD $ are categories. We will denote the structure maps the same way and leave the distinction implicit from the domains. A \emph{functor} $ \phi \colon \CC \to \CD $ consists of a pair of functions
  $\phi_1 \colon \CC_1 \to \CD_1$, $\phi_0 \colon \CC_0 \to \CD_0$
  with the following properties:
  \begin{itemize}\itemsep=0pt
    \item (Compatibility with source and target):
    $s \circ \phi_1 = \phi_0 \circ s$, $t \circ \phi_1 = \phi_0 \circ t$.
    \item (Compatibility with multiplication):
    $\forall (a,b) \in \CC_1 \times_{s,t} \CC_1 \to \CC_1$, $ \phi_1(a \bullet b) = \phi_1(a) \bullet \phi_1(b) $.
  \end{itemize}
  A functor can be visualized as a square
  \[\begin{tikzcd}
    \CC_1 \darrow \arrow[r, "\phi_1"] & \CD_1 \darrow \\
    \CC_0 \arrow[r, "\phi_0"] &\, \CD_0.
  \end{tikzcd}\]
\end{Definition}
A double category is a category internal to the category of categories. This is not the most useful definition for those who do not already know what such structures are.
\begin{Definition}
  A \emph{double category} $ \CD $ consists of four sets $\CC$, $\CK$, $\CH$ and $M$ together with four category structures which we visually arrange into a square
 \[\begin{tikzcd}
\CC \darrow \rarrow & \CK \darrow \\
\CH \rarrow &\, M.
\end{tikzcd}\]
The various structures are assumed to be \emph{compatible} in the sense that the source and target maps of $ \CC \grpd \CH $ and $ \CH \grpd M $ constitute functors from $ \CC \grpd \CK$ to $ \CH \grpd M $
  \[
  \begin{tikzcd}
    \CC \darrow \arrow[r, "s_1"] & \CK \darrow \\
    \CH \arrow[r, "s_0"] &\, M,
  \end{tikzcd}
  \qquad
  \begin{tikzcd}
    \CC \darrow \arrow[r, "t_1"] & \CK \darrow \\
    \CH \arrow[r, "t_0"] &\, M.
  \end{tikzcd}
  \]
The two multiplication operations on $ \CC $ must also constitute a pair of compatible binary operations. This is the algebraic rule encoding that the composition operation also constitutes a~homomorphism. If $ \circ $ is the binary operation for $ \CC \grpd \CK $ and $ \bullet $ is the binary operation for $ \CC \grpd \CH $, then for any four elements $ a,b,c,d \in \CC $ such that the following operations
$a \circ b$, $c \circ d$, $a \bullet c$, $b \bullet d$ are well defined,
we must have that the following equation is also well defined and holds
\begin{equation}\label{eqn:intertwininglaw}
(a \circ b) \bullet (c \circ d) = (a \bullet c) \circ (b \bullet d).
\end{equation}
\end{Definition}
Generally speaking, there are quite a large number of structure maps associated to a double category. We will generally try to avoid using them explicitly. Most of the time we will be concerned with the compatible pair of binary operations and we will often use the phrase ``whenever it is well defined'' as a shorthand for the appropriate composition rules.

However, in situations where we must make reference to these structures, we will observe the following convention: We call the category structures on $ \CC \grpd \CH $ and $ \CK \grpd M $ the \emph{horizontal categories} and the category structure maps will be denoted
$s^H_{i}$, $t^H_{i}$, $u^H_i$, $\bullet$, $i = 0,1$
where $ i = 1$ corresponds to $ \CC \grpd \CH $ and $ i= 0 $ corresponds to $ \CK \grpd M$. We will abuse notation and use $ \bullet $ to denote the horizontal composition for both groupoids.

The category structures on $ \CC \grpd \CK $ and $ \CH \grpd M $ are the \emph{vertical categories} and the category structure maps will be denoted
$s^V_{i}$, $t^V_{i}$, $u^V_i$, $\circ$, $i = 0,1$,
where $ i = 1$ corresponds to $ \CC \grpd \CK $ and~${i= 0}$ corresponds to $ \CH \grpd M$. We will use $ \circ $ to denote vertical composition for both groupoids.

It is often helpful to visualize elements $a \in \CC $ as a square
\[
\begin{tikzcd}
   x \arrow[rr, "s_1^V(a) \in \CK"] \arrow[dd, "s_1^H(a) \in \CH", swap] & & y \arrow[dd, "t_1^H(a) \in \CH"] \\
   & a & \\
   z\arrow[rr, "t_1^V(a) \in \CK", swap] & &\, w,
\end{tikzcd}
\]
where $ x,y,z,w \in M $ are the sources and targets of their respective arrows. The elements of $ \CC $ can be thought of as 2-cells.

The vertical and horizontal composition operations can be visualized by vertically or horizontally juxtaposing such squares. In this way, the compatibility law looks like
\[
\left( \begin{tikzcd}[column sep = tiny, row sep = tiny]
a \\
\circ \\
b
\end{tikzcd} \right)
\bullet
\left(
\begin{tikzcd}[column sep = tiny, row sep = tiny]
c \\
\circ \\
d
\end{tikzcd}\right)
=
\begin{tikzcd} [column sep = tiny, row sep = 0.15em]
\big( a &\bullet & c \big) \\
& \circ  & \\
\big( b & \bullet & d \big)
\end{tikzcd}.
\]

\subsection{Double groupoids}
Double groupoids are a type of double category that is of particular interest to us.
\begin{Definition}
A \emph{double groupoid} is a double category
 \[\begin{tikzcd}
\CG \darrow \rarrow & \CK \darrow \\
\CH \rarrow & \,M,
\end{tikzcd}
\]
where all of the category structures are in fact groupoids.
\end{Definition}
Let us consider a few examples.
\begin{Example}[groups]
  Let $ G $ be a group. Consider the following double groupoid:
\[\begin{tikzcd}
G \darrow \rarrow & G \darrow \\
\{*\} \rarrow & \,\{*\}.
\end{tikzcd}
\]
For the top and bottom groupoid structures, we take the trivial groupoid structure in which every arrow is a unit. For the left and right groupoid structures, we take the usual group operations. It is straightforward to check that these structures are compatible.
\end{Example}
\begin{Example}[strict 2-groupoids]
A \emph{strict $2$-groupoid} is a kind of double groupoid, where the bottom groupoid structure is trivial
\[
\begin{tikzcd}
  \CG_2 \darrow \rarrow & \CG_1 \darrow \\
\CG_0 \rarrow & \,\CG_0.
\end{tikzcd}
\]
Under the typical conventions for 2-categories, the set $ \CG_0 $ is the \emph{objects}, $ \CG_1$ is the \emph{arrows} and $ \CG_2$ is the \emph{$2$-arrows}.
\end{Example}
\begin{Example}[strict 2-groups]
A \emph{strict $2$-group} is a 2-groupoid with a single object. Hence, it can be interpreted as a double groupoid of the form
\[
\begin{tikzcd}
  G_2 \darrow \rarrow & G_1 \darrow \\
\{ * \} \rarrow &\, \{ * \}.
\end{tikzcd}
\]

\end{Example}
Lie double groupoids arise when we impose smoothness conditions on the structure maps. However, there is one slight caveat in that we require the so called ``double target'' map to be a submersion. In order to explain this, let us first consider the set
\[ \CG^\lrcorner := \big\{ (k,h) \in \CK \times \CH \colon t_0^H(k) = t_0^V(h) \big\}. \]
This set can be visualized as the set of bottom right ``corners'' of a double groupoid cell
\[
\CG^\lrcorner = \left \{ \begin{tikzcd}
 & x \arrow[d, "h"] \\
y \arrow[r, "k"] & z
\end{tikzcd} \right\}.
\]
Associated to this set is the double target map
$
t^D \colon \CG \to \CG^\lrcorner$, $a \mapsto \big(t_1^V(a), t_1^H(a)\big)$.
This operation maps an element of $ \CG $ to its bottom right corner
\[
\begin{tikzcd}
   x \arrow[rr, "k_1"] \arrow[dd, "h_1", swap] & & y \arrow[dd, "h_2"] \\
   & a & \\
   z\arrow[rr, "k_2", swap] & & w
\end{tikzcd} \quad
\mapsto \quad
\begin{tikzcd}
   & & y \arrow[dd, "h_2"]\\
   & & \\
   z \arrow[rr, "k_2"] & &\, w.
\end{tikzcd}
\]
\begin{Definition}\label{defn:lie.double.groupoid}
A \emph{Lie double groupoid} is a double groupoid
 \[\begin{tikzcd}
\CG \darrow \rarrow & \CK \darrow \\
\CH \rarrow &\, M,
\end{tikzcd}\]
where $ \CG$, $\CK$, $\CH$, $M $ are all equipped with smooth manifold structures which make the associated groupoid structures into Lie groupoids. We additionally require that the double target map~${t^D \colon \CG \to \CG^\lrcorner}$
is a surjective submersion.
\end{Definition}
The requirement that the double target map is a surjective submersion can be thought of as a kind of smooth filling condition.
An important consequence of this map being a submersion is that the set of vertically composable arrows in $ \CG $ is a \emph{smooth} Lie subgroupoid of the horizontal groupoid structure on $ \CG \times \CG $ (and vice versa). It also ensures that the space of composable squares $ \CG^\boxplus $, which is the domain of the compatibility law, is a smooth manifold
\begin{equation}\label{equation:groupoid.box}
\CG^\boxplus := \{ (a,b,c,d) \in \CG \mid a \circ b, \ c \circ d, \ a \bullet c, \ b \bullet d \ \text{are well defined} \}. \end{equation}

\section{Algebras associated to categories}

\subsection{Category algebras}
The simplest sort of algebra that we can attach to a category is the so-called category algebra.
\begin{Definition}\label{def:category.algebra}
  Let $ \CC $ be a category. The \emph{category algebra} of $ \CC $ is the vector space $ \BBC \CC $ generated freely by elements of $ \CC $.

   Associated to any element $ a \in \CC_1 $ let $ \delta_a \in \BBC \CC $ denote the associated basis vector. The algebra structure on $ \BBC \CC $ is defined in terms of basis elements
  \[ \forall a, b \in \CC_1, \qquad \delta_a * \delta_b := \begin{cases}
    \delta_{a \bullet b}, & s(a) = t(b), \\
    0, & s(a) \neq t(b).
  \end{cases} \]
  If $ \CC $ is a groupoid, this is called the \emph{groupoid algebra}. In the case where $ \CC $ is a group, then this is the classical group algebra.
\end{Definition}
One of the main distinctions between a category and its associated category algebra is that the product operation is globally defined.
\begin{Example}
   Suppose $ S $ is a finite set with $ n $ elements and $ \CG = S \times S \grpd S $ is the pair groupoid. Then the group algebra of $ \CG $ is isomorphic to $ M_{n \times n}(\BBC) $ the algebra of $ n $ by $ n $ matrices.

   To see why, let us chose a way to index $ S $ by natural numbers: $ S = \{ s_i \}_{1 \le i \le n} $. Then consider the isomorphism given by the following map of basis elements $\forall\ 0 \le i,j \le n$, $\delta_{(s_i,s_j)} \mapsto E_{ij}$
   where $ E_{ij}$ is the elementary $ n \times n$ matrix with a single $ 1 $ in the $ (i,j)$th entry. We leave it to the reader to verify that this is indeed an isomorphism of algebras.
\end{Example}

\subsection{Convolution algebras}
The convolution algebra of a topological groupoid provides us with a simultaneous generalization of the notion of a group algebra and the algebra of continuous functions on a space. A~more thorough reference on the basics of Haar systems and convolution algebras in groupoids can be found in Williams~\cite{DanaWilliamsBook}.
We will stick to the setting of Lie groupoids but much of the following discussion is well defined for many (nice enough) topological groupoids by replacing ``smooth'' with ``continuous''. Throughout this section, given a manifold $ M $, we will use the notation~$ C_c^\infty(M)$ to denote the set of complex valued, compactly supported, smooth functions on $ M $.

Before we can define the convolution algebra, we require the groupoid analogue of a Haar measure.
\begin{Definition}
Suppose $ \CG_1 \grpd \CG_0$ is a Lie groupoid. A \emph{smooth $($left$)$ Haar system} on $ \CG_1 \grpd \CG_0 $ is collection $ \{ \mu_x \}_{x \in \CG_0 } $, where for each $ x \in \CG_0 $, $ \mu_x$ is a Radon measure on $ t\inv(x) $. We require that the family $ \{ \mu_x \}_{x \in \CG_0} $ satisfies the following properties:
\begin{enumerate}\itemsep=0pt
  \item[(1)] (Smoothness): For each function $ u \in C_c^\infty(\CG_1) $, the associated function
  \[ \CG_0 \to \BBR, \qquad x \mapsto \int_{t\inv(x)} u|_{t\inv(x)} {\rm d} \mu_x  \]
  is smooth.
  \item[(2)] (Left invariant): For each $ g \in \CG_1 $ with $ s(g) = x$ and $ t(g) = y $, the function
  \[ L_g \colon\ t\inv(x) \to t\inv(y), \qquad h \mapsto g \circ h \]
  is measure preserving.
\end{enumerate}
\end{Definition}
Given such a Haar system, we will write $ \int {\rm d} \mu \colon C^\infty_c(\CG_1) \to C^\infty_c(\CG_0) $ to denote the fiberwise integration map that arises as a consequence of property (1).
\begin{Example}
Suppose $ G$ is a Lie group. If we consider $ G $ to be a groupoid with object space~$ \{ * \} $, then a Haar system on $ G \grpd \{ * \} $ is the same thing as a Haar measure on $ G $.
\end{Example}
\begin{Example}
   Suppose $ \CG_1 \grpd \CG_0 $ is an \'etale Lie groupoid. In other words, the source and target maps are \'etale maps. Then the $ t $-fibers of $ \CG_1 $ are zero-dimensional and any scalar multiple of the counting measure on the $ t $-fibers constitutes a Haar system.
\end{Example}

\begin{Definition}
  Suppose $ u, v \in C_c^\infty(\CG_1) $. Then the \emph{convolution} of $ u $ with $ v $ is the function
  \[
  u * v \colon\ \CG_1 \to \BBC, \qquad u*v(g) := \int_{t\inv(t(g))} u(h) v\big(h\inv \circ g\big) {\rm d} \mu_{t(g)}(h).
  \]
  This defines a (possibly non-unital) associative algebra structure
  $C_c^\infty(\CG_1 ) \otimes C_c^\infty(\CG_1) \to C_c^\infty(\CG_1)$.
\end{Definition}
There is an alternative construction of the convolution operation that provides us with a~bit more insight than the usual formula. To start, we first observe that given $ g \in \CG_1 $ we can canonically identify $ m\inv(g)$ with $ t\inv(t(g)) $ by taking the map
$t\inv(t(g)) \to m\inv(g)$, ${h \mapsto \big(h, h\inv \circ g \big)}$.
This provides us with a smooth family of measures $ \mu_g $ on $ m\inv(g) $ for each $ g \in \CG_1 $. From this point of view, the convolution of $ u $ with $ v $ is obtained by viewing $ u \otimes v $ as a function on $ \CG_1 \times \CG_1 $, restricting it to the set of composable arrows, and integrating along the fibers of $ m $
\[\begin{tikzcd}
C_c^\infty(\CG_1 ) \otimes C_c^\infty(\CG_1) \arrow[r]& C_c^\infty(\CG_1 \times \CG_1) \arrow[r, "\text{restrict}"] & C_c^\infty( \CG_1 \times_{s,t} \CG_1) \arrow[r, "{\int {\rm d} \mu_g}"] & C_c^\infty(\CG_1).
\end{tikzcd}
\]
\begin{Example}\label{example:category.algebra.convolution}
   Suppose $ \CG_1 \grpd \CG_0 $ has the discrete topology. Take the Haar system on $ \CG_1 \grpd \CG_0 $ which is given by the counting measure. The set $ C_c^\infty(\CG_1)$ can be canonically identified with $ \BBC \CG_1 $ and this constitutes an isomorphism of the convolution algebra with the groupoid algebra.
\end{Example}
\begin{Example}
   Consider $ \BBR \grpd \{ * \} $ where we take the group structure on $ \BBR $ by addition. We can equip $ \BBR $ with the standard measure. Then the convolution of a function $ u $ with a function $ v $ is the classical convolution
  $u * v (x) = \int_{\BBR} u(y) v(x-y) {\rm d}y $.
\end{Example}

\begin{Example}
   Suppose $ K $ is a discrete group acting smoothly on a manifold $ M $. The action groupoid $ K \ltimes M \grpd M $ is an \'etale groupoid. Indeed, the source and target fibers of $ K \ltimes M $ can be identified with $ K $. Let us take the counting measure as the Haar system on $ K \ltimes M \grpd M $. Since our measure is discrete, integration over the target fibers is given by a sum.

  In this case, convolution is given by the formula
  \[ u * v (k,p) := \sum_{h \in K } u\big(h, \big(h\inv k\big) \cdot p \big) v\big(h\inv k, p\big).
   \]
  This sum is finite so long as $ u $ and $ v $ are compactly supported.
\end{Example}

\section{Convolution algebras of double groupoids}\label{sec5}
With the background out of the way, we will now introduce our main context. Throughout this section, we will be considering a double Lie groupoid of the form
 \begin{equation*}\label{eqn:haar.system.double.groupoid}
  \begin{tikzcd}
\CG \darrow \rarrow & \CK \darrow \\
\CH \rarrow &\, M.
\end{tikzcd}
\end{equation*}
As before, let us denote by $ \circ $ the groupoid product for $ \CG \grpd \CK $ and $ \CH \grpd M $ and use $ \bullet $ to denote the groupoid product for $ \CG \grpd \CH $ and $ \CK \grpd M $.
\subsection{Double target map}
Let us recall the set of bottom right corners
$\CG^\lrcorner := \big\{ (k,h) \in \CK \times \CH \mid t^H_0(k) = t^V_0(h) \big\}$
and the double target map
$t^D \colon \CG \to \CG^\lrcorner$, $a \mapsto \big(t_1^V(a), t_1^H(a)\big)$.
Recall that in the definition of a~double Lie groupoid (see Definition~\ref{defn:lie.double.groupoid}) we require that $ t^D $ is a submersion.

There are two natural actions of $ \CG $ on $ \CG^\lrcorner$. Given $ (k,h) \in \CG^\lrcorner $ and $ a \in \CG $ such that $ s^V_1(a) = k $ then
$a \circ (k,h) := \big( t^V_1(a), t^H_1(a) \circ h\big)$.
Similarly, given $ a \in \CG $ such that $ s^H_1(a) = h $ then
$a \bullet (k,h) := \big( t^V_1(a) \bullet k, t^H_1(a)\big)$.
Our next lemma observes that this action makes $ t^D $ equivariant with respect to the natural vertical and horizontal translation maps.

\begin{Lemma}
  Suppose $ a, b \in \CG $ are such that $ s^V_1(a) = t^V_1(b) $, then
  $t^D(a \circ b) = a \circ t^D(b) $.
  On the other hand, if $ s^H_1(a) = t^H_1(b)$, then
  $t^D(a \bullet b) = a \bullet t^D(b)$.
\end{Lemma}
The proof is a straightforward calculation, which we leave to the reader.
\begin{Corollary}
  Given $ a \in \CG $ and $ (k,h) \in \CG^\lrcorner $, then if $ s^V_1(a) = k $, the following map is a~diffeomorphism
  \[ L^V_a \colon \ \big(t^D\big)\inv(k,h) \to \big(t^D\big)\inv( a \circ (k,h)), \qquad b \mapsto a \circ b. \]
  Similarly, if we instead have that $ s^H_1(a) = h $, then
  \[ L^H_a \colon \ \big(t^D\big)\inv(k,h) \to \big(t^D\big)\inv( a \bullet (k,h)), \qquad b \mapsto a \bullet b \]
  is a diffeomorphism.
\end{Corollary}
\subsection{Double Haar systems}
We can now state our definition of a double Haar system.

\begin{Definition}\label{defn:double.haar.system}
A \emph{double Haar system} on $ \CG $ consists of the following data:
\begin{itemize}\itemsep=0pt
  \item For each $ (k,h) \in \CG^\lrcorner $ a Radon measure $ \mu_{(k,h)}^D $ on the fiber $ \big(t^D\big)\inv(k,h) $.
  \item A pair of Haar systems: $ \big\{\mu_x^\CK\big\}_{x \in M} $ and $\big\{\mu_x^\CH\big\}_{x \in M } $ on $ \CK \grpd M $ and $ \CH \grpd M $, respectively.
\end{itemize}
We require this data to satisfy the following properties:
\begin{enumerate}\itemsep=0pt
  \item[(1)] (Smooth) The family of measures varies smoothly. In other words, for all $ u \in C^\infty_c(\CG) $ the function
  \[ (k,h) \mapsto \int_{(t^D)\inv(k,h)} u(a) {\rm d} \mu_{(k,h)}^D (a)\]
  is smooth.
  \item[(2)] (Doubly invariant) For all $ a \in \CG$ and $ (k,h) \in \CK $, we have that if $ s^V_1(a) = k $, then
  \[ L^V_a \colon\ \big(t^D\big)\inv(k,h) \to \big(t^D\big)\inv(a \circ (k,h))\]
  is measure preserving. Similarly, if $ s^H_1(a) = h $, then
    \[ L^H_a \colon\ \big(t^D\big)\inv(k,h) \to \big(t^D\big)\inv(a \bullet (k,h))\]
    is measure preserving.
\end{enumerate}
\end{Definition}
Note that property 1 implies that fiberwise integration along $ t^D $ induces a linear map
$\int {\rm d} \mu^D \colon\allowbreak C^\infty_c(\CG) \to C^\infty_c(\CG^\lrcorner)$.

Now let us see how to construct a pair of compatible Haar systems for each of the product structures on $ \CG $ out of a double Haar system. To this end, first let us observe that the set $ \CG^\lrcorner $ fits into a pullback diagram
  \[
  \begin{tikzcd}
    \CG^\lrcorner \arrow[r, "\pr_1"] \arrow[d, "\pr_2"] & \CK \arrow[d, "t^H_0"]\\
    \CH \arrow[r, "t^V_0"] &\, M.
  \end{tikzcd}
  \]
  Note that given any $ k \in \CK $ the fiber
  $(\pr_1)\inv(k) = \big\{ (k,h) \in \CK \times \CH \mid t^H_0(k) = t^V_0(h) \big\} \subset \CG^\lrcorner$
  is canonically diffeomorphic to
  \smash{$\big(t^V_0\big)\inv\big( t^H_0(k)\big) \subset \CH $}.
   Therefore, given a Haar system, \smash{$ \big\{ \mu_x^\CH \big\}_{x \in M} $} on $ \CH $ one obtains smooth family of measures \smash{$ \big\{ \overline{\mu}^\CH_{k}\big\}_{k \in \CK}$} along the fibers of $ \pr_1 $ and they induce a fiberwise integration map
  \smash{$\int {\rm d}\overline{\mu}^\CH \colon C^\infty_c(\CG^\lrcorner) \to C^\infty_c(\CK)$},
  where for $ u \in C^\infty_c(\CG^\lrcorner) $ the function $\bigl( \int {\rm d}\overline{\mu}^\CH \bigr) u $ is the map
  \[
   k  \mapsto  \int_{(t_0^V)\inv(t_0^H(k))} u(k,h) {\rm d} \mu^\CH_{t_0^H(k)} (h).
  \]
   Similarly, given a Haar system \smash{$ \big\{\mu_x^\CK\big\}_{x \in M }$} on $ \CK $ one obtains a natural family of measures \smash{$ \big\{\overline{\mu}^\CK_h\big\}_{h \in \CH }$} along the fibers of $ \pr_2 $ with a fiberwise integration map
  $\int {\rm d}\overline{\mu}^{\CK} \colon C^\infty_c(\CG^\lrcorner) \to C^\infty_c(\CH) $,
    where for $ u \in C^\infty_c(\CG^\lrcorner) $ the function $\big( \int {\rm d}\overline{\mu}^\CH \big) u $ is the map
  \[ h  \mapsto   \int_{(t_0^H)\inv(t_0^V(h))} u(k,h) {\rm d} \mu^\CK_{t_0^V(h)} (k).
  \]
  \begin{Lemma}\label{lemma:fiberwise.integral.lemma}
    The following diagram of fiberwise integrals commutes:
    \[
    \begin{tikzcd}
           C^\infty_c(\CG^\lrcorner) \arrow[d, "\int {\rm d}\overline{\mu}^\CK"] \arrow[r, "\int {\rm d}\overline{\mu}^\CH"] & C^\infty_c(\CK) \arrow[d, "\int {\rm d} \mu^\CK"] \\
      C^\infty_c(\CH) \arrow[r, "\int {\rm d} \mu^\CH"] & \,C^\infty_c(M).
    \end{tikzcd}
    \]
  \end{Lemma}
  \begin{proof}
    The claim is an immediate consequence of Fubini's theorem.
    Let $ u \in C^\infty_c(\CG^\lrcorner) $. Then computing the expressions
      $\left( \int {\rm d} {\mu}^\CH \right) \left( \int {\rm d}\overline{\mu}^\CK  \right)u$ and $ \left( \int {\rm d} \mu^\CK\right) \left( \int {\rm d}\overline{\mu}^\CH  \right)u$
    yields the functions
    \[ x  \mapsto  \int_{(t^V_0)\inv(x)} \left[ \int_{(t^H_0)\inv(x)} u(k,h) {\rm d} \mu_x^\CK(k) \right] {\rm d} \mu_x^\CH(h)\]
    and
    \[  x  \mapsto  \int_{(t^H_0)\inv(x)} \left[ \int_{(t^V_0)\inv(x)} u(k,h) {\rm d} \mu_x^\CH(h) \right] {\rm d} \mu_x^\CK(k),\]
    respectively. By Fubini's theorem, these two functions are equal.
  \end{proof}

\begin{Theorem}
   Suppose \smash{$ \big\{\mu_{(k,h)}^D\big\}_{(k,h) \in \CG^\lrcorner } $}, \smash{$ \big\{\mu_{x}^\CH \big\}_{x \in M }$} and \smash{$ \big\{ \mu_x^\CK \big\}_{x \in M} $} constitute a double Haar system as in Definition~{\rm\ref{defn:double.haar.system}}. There exist unique Haar systems $ \{\mu_{k}^\circ\}_{k \in \CK } $ and $ \{ \mu_{h}^\bullet \}_{h \in \CH} $ with respect to the vertical, $ \circ $, and horizontal, $ \bullet $, products on $ \CG$ which are uniquely determined by the property that they make the following diagram of fiberwise integrals commutes:
  \begin{equation}\label{diagram:double.haar.theorem}
    \begin{tikzcd}
    C^\infty_c(\CG) \arrow[ddr, bend right, "\int {\rm d} \mu^\bullet"] \arrow[rrd, bend left, "\int {\rm d} \mu^\circ"] \arrow[dr, "\int {\rm d} \mu^D"] & & \\
          & C^\infty_c(\CG^\lrcorner) \arrow[d, "\int {\rm d}\overline{\mu}^\CK"] \arrow[r, "\int {\rm d}\overline{\mu}^\CH"] & C^\infty_c(\CK) \arrow[d, "\int {\rm d} \mu^\CK"] \\
     &  C^\infty_c(\CH) \arrow[r, "\int {\rm d} \mu^\CH"] &\, C^\infty_c(M).
    \end{tikzcd}
  \end{equation}
  \end{Theorem}
  \begin{proof}
       By the Riesz-–Markov–-Kakutani representation theorem, there must exist a unique family of measures $ \{ \mu_k^\circ \}_{k \in \CK} $ on the fibers of $ t^V_1 $ with the property that it makes the following diagram commutes:
    \[
    \begin{tikzcd}
         C^\infty_c(\CG) \arrow[rrd, bend left, "\int {\rm d} \mu^\circ"] \arrow[dr, "\int {\rm d} \mu^D"] & & \\
      & C^\infty_c(\CG^\lrcorner)  \arrow[r, "\int {\rm d}\overline{\mu}^\CH"] &\, C^\infty_c(\CK).
    \end{tikzcd}
    \]
    This means that for any $ u \in C^\infty_c(\CG) $ and $ k \in \CK $, we have that
    \[
    \int_{(t^V_1)\inv(k)} u(b) {\rm d} \mu^\circ_{k}(b) =
    \int_{(t^V_0) \inv(t^H_0(k))} \left[\int_{(t^D)\inv(k,h)} u(b) {\rm d} \mu^D_{(k,h)}(b) \right] \mu^\CH_{t^H_0(k)} (h).
    \]

       We claim that this measure is a Haar measure for the $ \circ $ composition operation on $ \CG $. Any such family of measures will clearly be smooth since fiberwise integration will map smooth functions to smooth functions.
    Therefore, the only remaining thing to show is that $ \mu^\circ $ is invariant under the vertical product.

    In other words, we must show that for all $ a \in \CG $ with $ s^V_1(a) = k $ and $ t^V_1(a) = k' $, then we have that the diffeomorphism
    \[ L^V_a \colon\ \big(t^V_1\big)\inv(k) \to \big(t^V_1\big)\inv(k'), \qquad b \mapsto a \circ b \]
    is measure preserving.
    In terms of integrals, this is equivalent to proving that for all $ u \in C^\infty_c(\CG) $, we have that
    \[
     \int_{(t^V_1)\inv(k)} u(a \circ b) {\rm d} \mu_k^\circ(b) = \int_{(t^V_1)\inv(k')} u( b) {\rm d} \mu_{k'}^\circ (b).
    \]
    To show this, let us assume that $ a $ is of the form
\[
a = \begin{tikzcd}
   x \arrow[rr, "k"] \arrow[dd, "h_1", swap] & & y \arrow[dd, "h_2"] \\
   & a & \\
   z\arrow[rr, "k'", swap] & & \,w,
\end{tikzcd}
\]
so the source of $ k $ is $ x $ and the target of $ k $ is $ y $.

Now given a function $ u \in C^\infty_c(\CG) $, from the definition of $ \mu^\circ $, we have that
       \[ \int_{(t^V_1)\inv(k)} u(a \circ b) {\rm d} \mu_k^\circ(b) = \int_{(t^V_0)\inv(y)} \left[ \int_{(t^D)\inv(k,h)} u(a \circ b) {\rm d} \mu_{(k,h)}^D(b) \right] {\rm d} \mu_{y}^\CH(h).
     \]
    On the other hand,
    \[
       \int_{(t^V_1)\inv(k')} u( b) {\rm d} \mu_{k'}^\circ (b) = \int_{(t^V_0)\inv(w)} \left[ \int_{(t^D)\inv(k', h)} u( b) {\rm d} \mu_{(k',h)}^D(b) \right] {\rm d} \mu_{w}^\CH(h).
    \]
       Since $ \big\{ \mu_p^\CH\big\}_{p \in M } $ is a Haar system and is invariant for the product in $ \CH $, we can do a substitution for the outside integral of right-hand side where we replace $ h $ with $ h_2 \circ h $ and $ w $ with $ y $, so we get
    \[
       \int_{(t^V_1)\inv(k')} u( b) {\rm d} \mu_{k'}^\circ (b) = \int_{(t^V_0)\inv(y)} \left[ \int_{(t^D)\inv(k', h_2 \circ h)} u(b) {\rm d} \mu_{(k', h_2 \circ h)}^D(b) \right] {\rm d} \mu_{y}^\CH(h).
    \]
       Since we have assumed that $ \mu^D $ is invariant under left translation, we can perform a substitution on the right-hand side where $ b $ is replaced by $ a \circ b $ and $ (k', h_2 \circ h)$ is replaced by $ (k,h) $ and we get
     \[
       \int_{(t^V_1)\inv(k')} u( b) {\rm d} \mu_{k'}^\circ (b) = \int_{(t^V_0)\inv(y)} \left[ \int_{(t^D)\inv(k,h)} u(a \circ b) {\rm d} \mu_{(k,h)}^D(b) \right] {\rm d} \mu_{y}^\CH(h)
    \]
    and so
    \[
     \int_{(t^V_1)\inv(k)} u(b) {\rm d} \mu_{k}^\circ(b) = \int_{(t^V_1)\inv(k')} u( a \circ b) {\rm d} \mu_{k'}^\circ (b).
    \]
  This shows that $ \mu^\circ $ is a Haar measure for $ \circ $. A symmetrical argument can be performed for $ \mu^\bullet $.

  By Lemma~\ref{lemma:fiberwise.integral.lemma} and the construction of $ \mu^\circ $ and $ \mu^\bullet $, it follows that \eqref{diagram:double.haar.theorem} commutes.
  \end{proof}

\begin{Example}[discrete double groupoids]
   Suppose $ \CG $ is countable. Then if we take the counting measures for $ \mu^D$, $ \mu^\CK$ and $ \mu^\CH $ this can easily be seen to be a double Haar system. The induced vertical and horizontal Haar systems on $ \CG $ are also just the counting measure.
\end{Example}
\begin{Example}[Lie groups]\label{example:Lie.groups}
Given a double groupoid of the form
  \[\begin{tikzcd}
G \darrow \rarrow & G \darrow \\
\{*\} \rarrow & \,\{*\}.
\end{tikzcd}
\]
Note that in this case $ \CG^\lrcorner \cong G $ and $ t^D \colon G \to G $ is just the identity map. Therefore, the fibers of~$ t^D $ are singletons. A standard Haar measure on $ G $ induces a double Haar system, where we take~$ \mu^D $ to be the trivial measure.
\end{Example}
\begin{Example}[strict 2-groups]
    \[\begin{tikzcd}
G_2 \darrow \rarrow & G_1 \darrow \\
\{*\} \rarrow & \,\{*\}.
\end{tikzcd}
\]
In this example, the horizontal target map is trivial so $ t^D = t \colon G_2 \to G_1 $. In this case, to define a double Haar system one must choose a Haar system on $ G_2 \grpd G_1 $ and a Haar measure on the group $ G_1 \grpd \{ * \} $. In order to satisfy the axioms of a double Haar system, the Haar system on~${ G_2 \grpd G_1}$ must be invariant under translation relative to the horizontal structure $ G_2 \grpd \{ * \} $. This will induce a Haar measure on the group structure $ G_2 \grpd \{ * \} $ by composing integration along fibers of $ t $ with integration along the Haar measure for $ G_1 $.
\end{Example}
\subsection{Compatibility for countable double groupoids}
Let us suppose that $ \CG $ is countable so it is equipped with the discrete topology. We saw previously that choosing counting measures induces a double Haar system in this case. Furthermore, in Example~\ref{example:category.algebra.convolution}, we saw that $ C_c^\infty(\CG) $ can be identified with the freely generated vector space~$ \BBC \CG $. Furthermore, under this isomorphism the convolution operations \smash{$ \stardot $} and \smash{$ \starcirc$} are just the corresponding groupoid algebra products.

Our first observation is just a slight modification of the Eckmann--Hilton lemma.
\begin{Proposition}
Suppose $ \CG $ is a countable double groupoid. Then the convolution operations \smash{$ \stardot $} and \smash{$ \starcirc $} constitute a compatible pair of binary operations if and only if the two groupoid products are equal, $ \bullet = \circ $.
\end{Proposition}
\begin{proof}
Since $ \CG $ is countable, let us index it by natural numbers for convenience
$\CG = \{ g_i \}_{i \in \BBN}$.
Now consider the subset $ U^V \subset \CG $ of elements which are units with respect to the vertical structure. Similarly, let $U^H \subset \CG $ be the subset of units with respect to the horizontal structure.

Now, for natural numbers $ n \in \BBN $, consider the following functions on $ \CG $:
\[
e_n^V \colon\ \CG \to \BBC, \qquad e_n^V := \sum\limits^{i<n}_{g_i \in U^V} \delta_{g_i},\qquad
e_m^H \colon\ \CG \to \BBC, \qquad e_m^H := \sum\limits^{i<m}_{g_i \in U^H} \delta_{g_i}.
\]
In other words, $ e^V_n $ is the sum of the delta functions for first $ n $ elements of $ \CG $ that are vertical units. Similarly for $ e^H_m$.

Recall that the product \smash{$ \delta_x \stardot \delta_y $} of two delta functions behaves as follows: If $ (x,y) $ is not a~composable pair with respect to $ \bullet $, then the product is zero. On the other hand, if $ (x,y) $ is a~composable pair, the product is \smash{$ \delta_x \stardot \delta_y = \delta_{x \bullet y} $}. Since the function $ e^V_n $ is a sum of delta functions for \emph{units}, we conclude that
\[
e^V_n \stardot e^H_m = \sum^{i < \min(m,n)}_{g_i \in U^V \cap U^H} \delta_{g_i}.
\]
In other words, the convolution of $ e^V_n $ with $ e^H_m$ is just a sum of delta functions which are units with respect to both structures.

Furthermore, observe that for all $ u \in C_c^\infty(\CG) $, there exists a natural number $ n $ such that
\smash{$e^V_n \starcirc u = u \starcirc e^V_n = u$}.
Similarly for $ e^H_m$ and \smash{$ \stardot $}.

Now consider the equation
\begin{equation}\label{equation:idontknowwhat}
\big(e_n^V \starcirc e_m^H\big) \stardot \big( e_m^H \starcirc e^V_n\big)  = \big(e_n^V \stardot e_m^H\big) \starcirc \big( e_m^H \stardot e^V_n\big).
\end{equation}
From our previous observation, the result of the above computation will be a sum of delta functions for elements which are units with respect to both product structures. On the other hand, if $ n $ is large enough relative to $m$, we have that the left-hand side of equation~\eqref{equation:idontknowwhat} becomes
\[
\big(e_n^V \starcirc e_m^H\big) \stardot \big( e_m^H \starcirc e^V_n\big) = e_m^H \stardot e^H_m = e^H_m.
\]
This implies that every one of the constituent delta functions of $ e^H_m $ comes from a unit for the vertical composition. A symmetrical argument coming from computing the right-hand side of equation~\eqref{equation:idontknowwhat} implies that each of the delta functions for $ e^V_n $ come from elements that are units with respect to horizontal composition. In other words, the units for vertical and horizontal convolutions are the same and so~${e_n := e^H_n = e^V_n}$.
To finish the proof, we now consider that for a natural number $ n $ and element~${u,v \in C_c^\infty(\CG)}$ we have that
\[
\big(u \starcirc e_n\big) \stardot \big(e_n \starcirc v\big)  = \big(u \stardot e_n \big) \starcirc \big( e_n \stardot v\big).
\]
For $ n $ sufficiently large, it follows that
\[ u \stardot v = \big(u \starcirc e_n\big) \stardot \big(e_n \starcirc v\big)  = \big(u \stardot e_n \big) \starcirc \big( e_n \stardot v\big) = u \starcirc v.
\]
Hence \smash{$ \stardot$} and \smash{$ \starcirc$} are equal. Since \smash{$ \starcirc $} and \smash{$ \stardot $} are groupoid algebras, by computing \smash{$ \starcirc $} and \smash{$ \stardot $} on basis elements, it follows that~${\bullet = \circ}$.
\end{proof}

We consider it likely that a version of the above proof exists for convolution algebras of double Lie groupoids. However, we will not include a proof of the more general case here.

\section{Noncommutative torus}\label{Section:noncommutative.torus}
We will now take a look at a slightly more complicated but quite important example: the noncom\-mutative torus. Let $ r \in \BBR $ be a fixed real number.
We consider the circle group $ \BBS^1 $ as a~quotient of $ \BBR $ by the subgroup $ 2 \pi \BBZ $. Now, consider the group homomorphism
$\phi_r \colon \BBZ \to \BBS^1$, $\phi_r(n) = [rn]_{2 \pi}$.
With this data, one can construct a double groupoid (indeed a 2-group) of the form\looseness=-1
 \[\begin{tikzcd}
\BBZ \ltimes \BBS^1 \darrow \rarrow & \BBS^1 \darrow \\
\{ * \} \rarrow &\, \{ * \}.
\end{tikzcd}
\]
As a set, $ \BBZ \ltimes \BBS^1 $ is just the standard Cartesian product.
The semi-direct product notation is used due to the fact that the groupoid structure on $ \BBZ \ltimes \BBS^1 \grpd \BBS^1 $ is the action groupoid associated to the homomorphism $ \phi_r \colon \BBZ \to \BBS^1 $.

In other words, the source and target are given by
$s(n,\theta) = \theta$, $ t(n,\theta) = rn + \theta$.
If we have~$ (m,\psi) $ and $ (n,\theta)$ composable, then
$(m,\psi) \circ (n, \theta) = (n + m, \psi)$.
Finally, the inverse map is
$i(n,\theta) = (-n, rn + \theta) $.

The group structure on $ \BBS^1 \grpd \{ * \} $ is just the standard circle group and the group structure on $ \BBZ \ltimes \BBS^1 \grpd \{ * \} $ is the one obtained by regarding it as a simple Cartesian product of groups.

As usual, we will refer to $ \circ $ as the \emph{vertical composition} and $ \bullet $ as the \emph{horizontal composition}.

\subsection{Orthonormal basis}
The space of functions on $ \BBS^1 $, with Haar measure $\mu$ that has been normalized \big(so that $\mu\big(\BBS^1\big)=1$\big), has a particularly nice form. In particular, it admits a nice countable basis indexed by the integers. Since $ \BBZ \ltimes \BBS^1 $ is just the product space, any compactly supported function ${f\in C_c^\infty\big(\BBZ \ltimes \BBS^1 \big)}$ can be written as a sum of functions of the form $g \otimes h$ where $g\in C_c^\infty(\BBZ)$ and $h\in C^\infty\big(\BBS^1\big)$. We will take advantage of this basis to investigate the relationship between \smash{$ \stardot $} and \smash{$ \starcirc $}.

For each $ k \in \BBZ $, let $e_k(\theta)={\rm e}^{{\rm i} k\theta}$. The collection of functions $ \{ e_k \}_{k \in \BBZ} $ constitutes an orthonormal \big(relative to the $L^2$ inner product\big) basis for $ C^\infty\big(\BBS^1\big) $.

We can use these functions to construct a related basis for $ C_c^\infty\big(\BBZ \ltimes \BBS^1\big) $. Consider the collection of functions $\{u_{jk}\}_{j,k\in\BBZ} \subset C_c^\infty\big(\BBZ \ltimes \BBS^1 \big)$, where
\begin{equation}\label{ONB-torus}
u_{jk}(n,\theta) := \begin{cases} {\rm e}^{{\rm i} k\theta} & \text{if } n=j,\\
0 & \text{otherwise}.\end{cases}
\end{equation}
In other words, the index $ j $ refers to the level of $ \BBZ $ where $ u_{jk}$ is supported and the index $ k $ refers to the frequency of $ u_{jk} $.

\subsection{Convolution algebras}
Note that the source fibers of $ \BBZ \ltimes \BBS^1 \grpd \BBS^1 $ are discrete, so for our vertical Haar system $ \mu^V $, we simply take the counting measure. For a Haar measure on $ \BBZ \ltimes \BBS^1 \grpd \{ * \} $, we simply take $ \mu^H $ to be the product of the counting measure on $ \BBZ $ with the normalized Haar measure on $ \BBS^1$.

Utilizing the definition of the convolution algebras yields that, for any pair of compactly supported functions $u,v\in C_c^\infty\big(\BBZ \ltimes \BBS^1\big)$, we have
\begin{align*} 
\begin{aligned}
&u\stardot v(n,\theta)= \sum_{m \in \BBZ} \int_{\phi \in \BBS^1} u(m,\phi)v(n-m,\theta-\phi){\rm d} \phi,\\
&u\starcirc v(n,\theta)= \sum_{m\in\BBZ} u(m,r(n-m)+\theta)v(n-m,\theta).
\end{aligned}
\end{align*}
Recall that the real number $ r \in \BBR $ refers to the constant determining the action of $ \BBZ $ on $ \BBS^1 $. Note that one of these convolutions includes an integral while the other only includes a sum. Since we assume that $ u $ and $ v $ are compactly supported, the sums are finite.

\begin{Proposition}\label{prop:nct.circ.calculation} Let $u_{ab}$ and $u_{cd}$ be orthonormal basis elements as defined in {\rm\eqref{ONB-torus}}. Then
\begin{equation}\label{nct-circ-conv}
u_{ab}\starcirc u_{cd}={\rm e}^{{\rm i} r bc} u_{(a+c)(b+d)}.
\end{equation}
\end{Proposition}

\begin{proof}
We have
\begin{align}
   u_{ab}\starcirc u_{cd}(n,\theta) & = \sum_{m\in \BBZ} u_{ab}(m,r(n-m)+\theta)u_{cd}(n-m),\theta)\nonumber\\
  & = u_{ab}(a,r(n-a)+\theta)u_{cd}(n-a),\theta).\label{expression-starcirc}
\end{align}
Note that if $c\neq n-a$ then the expression \eqref{expression-starcirc} just becomes zero. If $c=n-a$, then the expression \eqref{expression-starcirc} becomes
$
{\rm e}^{{\rm i}rbc}{\rm e}^{{\rm i} b\theta}{\rm e}^{{\rm i} d\theta}={\rm e}^{{\rm i} rbc} {\rm e}^{{\rm i}(b+d)\theta}$.
So
\[
u_{ab}\starcirc u_{cd}(n,\theta)=\begin{cases} {\rm e}^{{\rm i} rbc} u_{(a+c)(b+d)} & \text{if } n=a+c,\\
0 & \text{otherwise}
\end{cases}
\]
thus proving \eqref{nct-circ-conv}.
\end{proof}

The above calculation is the standard method for recovering the classical ``noncommutative torus'' algebra out of the action groupoid associated to a homomorphism $ \BBZ \to \BBS^1 $.

Now we compute the other convolution in this basis.
\begin{Proposition} Let $u_{ab}$ and $u_{cd}$ be orthonormal basis elements as defined in \eqref{ONB-torus}. Then
\begin{equation*}\label{nct-dot-conv}
u_{ab}\stardot u_{cd}=\begin{cases} u_{(a+c)b} & \text{if } b=d,\\
0 & \text{otherwise}.
\end{cases}
\end{equation*}
\end{Proposition}

\begin{proof}
We have
\begin{align}
  \notag u_{ab}\stardot u_{cd}(n,\theta) & = \sum_{m\in\BBZ} \int_{\phi\in \BBS^1} u_{ab}(m,\phi)u_{cd}(n-m,\theta-\phi){\rm d}\phi\\
  & \label{conv-integral} = \int_{\phi\in\BBS^1} u_{ab}(a,\phi)u_{cd}(n-a,\theta-\phi){\rm d}\phi.
\end{align}
If $c=n-a$ (thus $n=a+c$), then the integral \eqref{conv-integral} becomes
\begin{equation}\label{int-proof}
 \int_{\BBS^1} {\rm e}^{{\rm i} b\phi}{\rm e}^{{\rm i} d(\theta-\phi)}{\rm d}\phi.
 \end{equation}
 If $b=d$, then \eqref{int-proof} becomes
$ {\rm e}^{{\rm i} d\theta} \int_{\BBS^1} 1 {\rm d}\phi
= {\rm e}^{{\rm i} d\theta}
= u_{(a+c)d}(n,\theta)
$
(since the measure on $\BBS^1$ has been normalized). Otherwise (if $b\neq d$ or $n\neq a+c$), we have zero.
\end{proof}

\subsection{Compatibility behavior of the two convolutions}
We observed earlier that for countable double groupoids, the compatibility of the convolution products appears to depend on multiplying ``composable'' elements. We will see that, under our basis, composability is determined by the \emph{frequency} component.

We consider the following two expressions:
\begin{equation}\label{comp-one}
\big(u_{ab}\starcirc u_{cd}\big)\stardot\big(u_{ef}\starcirc u_{gh}\big)
\end{equation}
and
\begin{equation}\label{comp-two}
\big(u_{ab}\stardot u_{ef}\big)\starcirc\big(u_{cd}\stardot u_{gh}\big).
\end{equation}

\begin{Proposition} The two expressions \eqref{comp-one} and \eqref{comp-two} are equal whenever $b=f$ and $d=h$. It also holds whenever $b+d\neq f+h$.
\end{Proposition}

\begin{proof}
  The expression \eqref{comp-one} gives us
  \begin{align*}
     \big({\rm e}^{{\rm i} rbc}u_{(a+c)(b+d)}\big)\stardot\big({\rm e}^{{\rm i} rfg}u_{(e+g)(f+h)}\big)&
    = {\rm e}^{ r(bc+fg)}\big(u_{(a+c)(b+d)}\stardot u_{(e+g)(f+h)}\big)\\
    & = \begin{cases} {\rm e}^{{\rm i} r(bc+fg)} u_{(a+c+e+g)(b+d)} & \text{if } b+d=f+h,\\
    0 &\text{otherwise}.
    \end{cases}
  \end{align*}
  On the other hand, if $b\neq f$ or $d\neq h$, then \eqref{comp-two} is just zero. So if $b=f$ and $d=h$, then the expression \eqref{comp-two} becomes
\smash{$
     u_{(a+e)b} \starcirc u_{(c+g)d}
    = {\rm e}^{{\rm i} b(c+g)} u_{(a+e+c+g)(b+d)}$}.
  That is, expression \eqref{comp-two} is
  \[
  \begin{cases} {\rm e}^{{\rm i} b(c+g)} u_{(a+e+c+g)(b+d)} & \text{if } b=f\text{ and }d=h,\\
  0 & \text{otherwise}.
  \end{cases}
  \]
   From here, it is easy to see that expression \eqref{comp-one} becomes \eqref{comp-two} if $b=f$ and $d=h$. We also see that \eqref{comp-one} and \eqref{comp-two} are zero whenever $b+d\neq f+h$.
\end{proof}

\section{Lie groups}\label{Section:compact.Lie.groups}
Let us now proceed to a particularly simple example of a double groupoid. This case is a stepping stone for studying the more complicated case of compact singular Lie groups.

Any Lie group $ G $ can be made into a Lie double groupoid as below
\[
\begin{tikzcd}
 G \darrow \rarrow & G \darrow \\
\{* \} \rarrow &\, \{ * \}.
\end{tikzcd}
\]
The groupoid structure on $ G \grpd G $ is trivial with only identity arrows. Hence, the two ``product'' operations on $ G $ are given by the formulas
$g \circ g = g$, $g \bullet h = gh$.
Where $ gh$ refers to the group operation.

In Example~\ref{example:Lie.groups}, we saw that one can construct a canonical double Haar system on this groupoid out of a Haar measure $ \mu $ on $ G $. The resulting Haar system on the top groupoid $ G \grpd G $ is a family of counting measures (the target fibers are singletons). On the left groupoid $ G \grpd \{ * \} $ the Haar system is just $ \mu $ itself (the only target fiber is $ G $). Under these conventions, it is easy to see that \smash{$ \starcirc $} is just the usual pointwise multiplication of functions on $ G $ and \smash{$ \stardot $} is the usual convolution operation for the group
\[ \big(u \starcirc v\big) (g) = u(g) v(g), \qquad \big(u \stardot v \big) (g) = \int_{h \in G } u(h) v\big(h\inv g\big) {\rm d} \mu.\]
Even in this very simple case, the question of how these two product operations are related is not so trivial.

From now on, to simplify the notation we will denote the pointwise product \smash{$ \starcirc $} by juxtaposition by $ \cdot $ and the convolution \smash{$ \stardot $} with just $ * $.

\subsection{Matrix coefficients}\label{subsection:matrix.coefficients}

Matrix coefficients are a special class of functions on $ G $. They are the functions one obtains by looking at the coefficients of a matrix representation of $ G $. Since there is a close relationship between the representation theory of $ G $ and its convolution algebra, it is not too surprising that matrix coefficients of representations of $ G $ exhibit some special behavior under convolution.
\begin{Definition} A \textit{finite-dimensional representation} $(\pi,V_\pi)$ is a continuous homomorphism ${\pi\colon G\rightarrow \GL(V_\pi)}$ from $G$ into the linear automorphisms of a finite-dimensional complex space~$V_{\pi}$.
\end{Definition}
Given a representation, $ (\pi, V_\pi) $, we will use $ d_\pi $ to denote the dimension of the underlying vector space $ V_\pi $. For convenience, we will always assume that $ V_\pi $ comes equipped with an ordered basis $ \{ e^\pi_1, \ldots, e^\pi_{d_\pi} \} $. Such an ordered basis induces an inner product on $ V_\pi $.

When $ G $ is compact, we will additionally make the following assumptions (thus making $\pi$ a~\textit{unitary} representation):
\begin{itemize}\itemsep=0pt
  \item the underlying vector space $ V_\pi $ is equipped with a Hilbert space structure;
  \item the basis $ \{ e^{\pi}_1, \ldots, e^{\pi}_{d_\pi} \} $ is orthonormal;
  \item the image of $ \pi $ lies in the unitary group $ U(V_\pi) $.
\end{itemize}
For compact groups, these assumption are quite reasonable as it is always possible to find such an invariant inner product via averaging.
\begin{Definition} Given a representation $ (\pi, V_\pi) $, if $u$, $v$ are vectors in $V_\pi$, we define the \textit{matrix coefficient for $ (u,v) $} to be
$\pi_{uv} \colon G \to \mathbb{C}$, $\pi_{uv}(g) :=\langle u,\pi(g)v\rangle$.
\end{Definition}
\begin{Remark}
The Peter--Weyl theorem\footnote{The Peter--Weyl theorem also states that irreducible unitary representations of a compact Lie group are always finite-dimensional.} tells us that, if $ G $ is compact, then matrix coefficients of \emph{irreducible} unitary representations form an orthonormal basis for a dense subspace of $ L^2(G) $. In particular, matrix coefficients span an $ L^2$-dense subspace of $ C^\infty(G) $. We bring up this fact to motivate our point of view that computations involving matrix coefficients are, in some sense, ``general''.
\end{Remark}
\begin{notation}\label{Notation:matrix.coefficients}
We will focus on matrix coefficients that come from basis vectors so it is convenient to establish some notation for them. Given a finite-dimensional representation $ (\pi, V_\pi ) $ equipped with an ordered basis $ \{ e^\pi_1,\ldots, e^\pi_{d_{\pi}} \} $, then let
$1 \le i,j \le d_\pi$, $ \pi_{ij} \colon G \to \mathbb{C}$, $ \pi_{ij}(g):= \pi_{e_i e_j} = \langle e_i^{\pi},\pi(g)e_j^{\pi}\rangle$.
\end{notation}

\subsection{Compatibility equations}
Our overall objective is to obtain equations which are ``close'' to a compatibility law. We know that the naive compatibility law for $ * $ and $ \cdot $ cannot hold in general but our study so far suggests that some compatibility laws can hold so long as constraints are placed on which products are permitted.

Our next proposition examines how one side of the usual compatibility law interacts with the matrix coefficients of a representation.
 \begin{Proposition}\label{proposition:convolution.of.non.irreducible.coefficients}
  Let $ G $ be a Lie group and let $ \pi $ be a finite-dimensional representation of $ G $. Given arbitrary compactly supported smooth functions $ u, v \in C_c^\infty(G) $ and
  $1 \le i, k \le d_\pi $,
  we have that
$\sum_{j=1}^{d_\pi} (u \cdot \pi_{ij}) * ( v \cdot \pi_{jk}) = (u * v) \cdot \pi_{ik}$.
\end{Proposition}
\begin{proof}
Note that by the formula for matrix multiplication, for any $ g,h \in G $ we have that
$\pi_{ik}(gh) = \sum_{j=1}^{d_\pi} \pi(g)_{ij} \pi(h)_{jk}$.
Keeping this formula in mind, we can prove the proposition by a~direct calculation
\begin{align*}
   \sum_{j=1}^{d_\pi} ( u \cdot \pi_{ij} ) * ( v \cdot \pi_{jk} ) (g) &= \sum_{j=1}^{d_\pi} \int_{G} u(h) \pi(h)_{ij} v\big(h\inv g\big) \pi\big(h\inv g\big)_{jk} {\rm d} \mu(h) \\
  &= \int_{G} u(h) v\big(h\inv g\big) \left( \sum_{j=1}^{d_\pi} \pi_{ij}(h) \pi_{jk}\big(h\inv g\big) \right) {\rm d} \mu(h) \\
   &= \int_G u(h) v\big(h\inv g\big) \pi_{ik}(g) {\rm d} \mu(h)  \\
   &= \left( \int_G u(h) v\big(h\inv g\big) {\rm d} \mu(h)\right) \pi_{ik}(g) = ( u * v)(g) \cdot \pi_{ik}(g).\tag*{\qed}
\end{align*}\renewcommand{\qed}{}
\end{proof}

Note that the above proposition only relies on the fact that $ \pi $ is a finite-dimensional representation. The group $ G $ does not necessarily need to be compact. If we additionally assume $ G $ is compact, then we have the following as a special case of our lemma.
\begin{Corollary}\label{corollary:conv.of.non.irred}
  Suppose $ G $ is compact and the Haar measure is normalized so that ${\int_G 1 {\rm d} \mu = 1}$. If $ \pi $ is a finite-dimensional representation of $ G $, then for all
  $1 \le i, k \le d_\pi$,
  we have that \smash{$\sum_{j=1}^{d_\pi} \pi_{ij} * \pi_{jk} = \pi_{ik}$}.
\end{Corollary}
\begin{proof}
   Since the Haar measure is normalized, $ 1 * 1 = 1 $. If we apply Proposition~\ref{proposition:convolution.of.non.irreducible.coefficients} to the case~${ u = 1}$, $v = 1 $, we immediately obtain
  \[ \sum_{j=1}^{d_\pi} \pi_{ij} * \pi_{jk} = \sum_{j=1}^{d_\pi} (1 \cdot \pi_{ij}) * (1 \cdot \pi_{jk}) = (1 * 1) \cdot \pi_{ik} = \pi_{ik}.\tag*{\qed}
   \]\renewcommand{\qed}{}
\end{proof}

Now let us state our main observation for this section, which is a kind of compatibility law. The caveat being that the compatibility appears to hold ``on average'' and requires that some terms come from a representation of $ G $.
\begin{Proposition}\label{proposition:new.weak.compatibility}
Let $ G $ be a compact Lie group and suppose $ \pi $ is a finite-dimensional representation of $ G $. For all
  $0 \le i,k \le d_\pi $
  and functions $ u, v \in C^\infty(G) $, we have that
\begin{equation*}\label{eqn:thm.weakly.hopeful}
  \sum_{j=1}^{d_\pi} (u \cdot \pi_{ij}) * ( v \cdot \pi_{jk} ) = \sum_{j=1}^{d_\pi} (u * v) \cdot (\pi_{ij} * \pi_{jk} ).
  \end{equation*}
\end{Proposition}
\begin{proof}
The proposition follows immediately from combining Corollary~\ref{corollary:conv.of.non.irred} with Proposition~\ref{proposition:convolution.of.non.irreducible.coefficients},
 \begin{align*}
 \sum_{j=1}^{d_\pi} ( u \cdot \pi_{ij}) * ( v \cdot \pi_{jk}) &= (u * v) \cdot \pi_{ik} = (u * v) \cdot \left( \sum_{j=1}^{d_\pi} \pi_{ij} * \pi_{jk} \right) \\
 &= \sum_{j=1}^{d_\pi} (u * v) \cdot ( \pi_{ij} * \pi_{jk} ). \tag*{\qed}
 \end{align*}\renewcommand{\qed}{}
\end{proof}

\subsection{Further remarks on the compact Lie group case}\label{subsection:further.remarks}
However, let us now provide a few different formulations and addition remarks on Proposition~\ref{proposition:new.weak.compatibility}. First, consider a special case: Given a representation $ \sigma $ of $ G $ if we take $ u = \sigma_{ab} $ and $ v = \sigma_{ac} $, we obtain, as a special case, the following equation:
\begin{equation}\label{eqn:new.compatibility.in.basis}
  \sum_{j=1}^{d_\pi} ( \sigma_{ab} \cdot \pi_{ij} ) * (\sigma_{bc} \cdot \pi_{jk} ) = \sum_{j=1}^{d_\pi} (\sigma_{ab} * \sigma_{bc} ) \cdot (\pi_{ij} * \pi_{jk}).
\end{equation}
This version of the equation is of some interest since it expresses our compatibility law purely in terms of matrix coefficients. It also bears resemblance to the compatibility law for a disjoint unions of pair groupoids.
This suggests a possibly deeper relationship since the Peter--Weyl theorem tells us that the convolution algebra of a compact group is isomorphic to the convolution algebra of a disjoint union of pair groupoids.

With that in mind, it is tempting to wonder whether or not it is possible to omit the sum from equation~\eqref{eqn:new.compatibility.in.basis} to obtain something of the form
\begin{equation}\label{eqn:new.compatibility.in.basis.hopeful}
  (\sigma_{ab} \cdot \pi_{ij} ) * (\pi_{jk} \cdot \sigma_{bc}) = (\sigma_{ab} * \sigma_{bc} ) \cdot (\pi_{ij} * \pi_{jk}).
\end{equation}
This would be nice since it would provide us with a ``true'' compatibility law. Unfortunately,\ equation~\eqref{eqn:new.compatibility.in.basis.hopeful} is generally false. To see a counter-example, we suggest computing the convolution products and point-wise products of irreducible matrix coefficients for $ G = \SU(2) $.

Some simplifications can be made if we assume $ \pi $ and $ \sigma $ are irreducible, unitary representations. In such a case, it is a relatively simple consequence of the Schur orthogonality that
$ \pi_{ij} * \pi_{jk} = \frac{1}{d_\pi} \pi_{ik}$
holds for irreducible unitary representations $ \pi $.

However, even if we assume $ \pi $ and $ \sigma $ are irreducible, equation~\eqref{eqn:new.compatibility.in.basis.hopeful} is false, as we can see in the $ G= \SU(2)$ case. The obstruction arises from examining how point-wise multiplication behaves for matrix coefficients. The product of two matrix coefficients turns out to be a matrix coefficient for the tensor product representation. However, it is a fact that tensor products of irreducible representations are not necessarily irreducible.\footnote{In fact, since matrix coefficients of irreducible representations form a basis, it is still possible to write the matrix coefficients of the tensor product representation $\pi\otimes\sigma$ in terms of the irreducible representations. However, to compute the coefficients explicitly is rather complex in practice. See the Littlewood--Richardson rule in~\cite{KnappBeyond} for an approach.} If we assume $ \pi$, $\sigma $ and $ \sigma \otimes \pi$ are all irreducible, then equation~\eqref{eqn:new.compatibility.in.basis.hopeful} holds. For example, if $ G $ is a torus, then tensor products of irreducible representations are indeed irreducible (but they are also all one-dimensional). Ultimately, this case is not so helpful as we are not aware of any example where $ \sigma \otimes \pi $ is both irreducible and has dimension greater than one.

\section{Compact singular Lie groups}\label{Section:compact.singular.Lie.groups}
We will now look at some examples of 2-groups where we can carry out a generalization of the computation we did for the noncommutative torus. We refer to these examples as ``compact singular Lie groups''. As singular spaces, they represent quotients of compact Lie groups by (possibly non-closed) discrete normal subgroups.

Throughout this section, let $ G $ be a compact connected Lie group and let $ K \subset G $ be a~countable, normal subgroup. We do not assume that $ K $ is closed. An immediate consequence of $ K $ being countable and normal (and $ G $ being connected) is that $ K $ is a subgroup of the center of~$ G $. There is a natural action of $ K $ on $ G $ via the inclusion and therefore we can form an action groupoid
$K \ltimes G \grpd G $.
As a reminder, the structure maps for this groupoid are as follows:
\[ s(\kappa, g) = g, \qquad t(\kappa, g) = \kappa g,\qquad
 (\kappa_1, \kappa_2 g) \circ (\kappa_2, g) = (\kappa_1 \kappa_2, g),  \qquad
 i(\kappa,g) = \big( \kappa^{-1}, \kappa g\big). \]
On the other hand, there is also the usual product group structure on $ K \ltimes G $,
$(\kappa_1, g_1) \bullet (\kappa_2, g_2) = (\kappa_1 \kappa_2, g_1 g_2) $.
These two composition operations come from a double groupoid of the form
\[
\begin{tikzcd}
K \ltimes G \darrow \rarrow & G \darrow \\
\{* \} \rarrow &\, \{ * \}.
\end{tikzcd}
\]
\begin{Remark}
   The compatibility of these two operations relies on the fact that $ K $ is a subgroup of the center of $ G $. If $ K $ is not countable, this construction still results in a Lie double groupoid so long as $ K $ is an immersed Lie subgroup of the center. However, we will focus on the countable case.
\end{Remark}

\begin{notation}
We are concerned with compactly supported functions on the product space $ K \times G $. Generally, given $ \kappa \in K $ we will write $ \delta_\kappa \in C^\infty(K) $ to denote the standard delta function for the element $ \kappa$. Since $ K $ is equipped with a discrete structure, these functions are ``smooth'' for our purposes.

Given sets $ A $ and $ B $ and functions $ u \colon A \to \BBC $, $ V \colon B \to \BBC $, we write $ u \otimes v $ to denote the function~${u \otimes v \colon A \times B \to \BBC}$, $ (u \otimes v)(a,b) = u(a) b(c)$.
We are primarily interested in functions on $ K \times G $, so we will frequently consider functions of the form $ \delta_\kappa \otimes u $ where $ u $ is a function on~$ G $ and $ \kappa $ is an element of $ K $. Such a function should be thought of as a function supported on the $ \kappa$-level of $ K \times G $.

In order shorten our equations somewhat, given $ \kappa \in K $ and a smooth function $ u \in C^\infty(G) $ we will use superscripts
$u^\kappa := \delta_{\kappa} \otimes u $
as a shorthand.
\end{notation}
\subsection{Computing convolution products}
We can now establish some formulas for the convolution products of matrix coefficients on $K\ltimes G$. Our first formula is just an immediate consequence of the definition of the fact that $ \bullet $ is just the product group structure on $ K \times G $.
\begin{Lemma}
  Given $ \kappa, \lambda \in K $ and functions $ u, v \in C^\infty(G) $, then we have that
  $ \big(u^\kappa \stardot v^\lambda \big) = (u * v )^{\kappa \lambda}$.
  Where, in the above formula, $ * $ denotes convolution of functions on $ G $.
\end{Lemma}
Slightly more complicated is the other product. In order to perform the calculation for~\smash{$ \starcirc $}, let us quickly establish some notation for the translation operator: Given a function ${f \in C_c^\infty(K \ltimes G)}$, there is a natural action of $ K $ by left translation. Given $ \kappa \in K $ let
$L_\kappa f \colon G \to \BBC$, $(L_\kappa f) (x) := f(\kappa x)$.
The translation operator interacts with convolution in $ G $ as follows.

\begin{Lemma}\label{lemma:translation.and.convolution}
  Let $ u, v \in C^\infty(G) $ and suppose $ \kappa \in G $. Then
  $L_{\kappa} (u * v) = (L_\kappa u) * v$.
  If $ \kappa \in Z(G) $ then we have that
  $ L_\kappa( u * v) = u * (L_{\kappa} v)$.
\end{Lemma}
\begin{proof}
  Let $x\in G$ be fixed. Then a simple change of variables $ h \mapsto k h $ yields
    \begin{align*}L_\kappa(u * v)(x) =  \int_G u(h)v\big(h^{-1}\kappa x\big){\rm d}h
    = \int_G u(\kappa h) v\big(h\inv x\big) {\rm d}h.
    \end{align*}
    For the second part, we use the fact that $ \kappa $ commutes with $ h $ to get
  \begin{align*}
  L_\kappa(u * v)(x) =  \int_G u(h)v\big(h^{-1}\kappa x\big){\rm d}h
    = \int_G u(h)v\big(\kappa h^{-1}x\big){\rm d}h=(u * L_\kappa v)(x).\tag*{\qed}
    \end{align*}\renewcommand{\qed}{}
\end{proof}

\begin{Lemma}\label{lemma:compute.star.circ}
  Given $ \kappa, \lambda \in K $ and functions $ u, v \in C^\infty(G) $, then we have that
  $u^\kappa \starcirc v^\lambda = ( L_{\lambda} u \cdot v )^{\kappa \lambda}$.
\end{Lemma}
\begin{proof}
The proof is a direct calculation. Let us compute $ u^\kappa \starcirc v^\lambda $ at a point $ (x,y ) \in K \times G $,
\begin{align*}
  \big[ (\delta_{\kappa} \otimes u ) \starcirc (\delta_{\lambda} \otimes v ) \big] (x,y) &=
  \sum_{\ell \in K } ( \delta_\kappa \otimes u) \big(\ell, \ell\inv x y\big) \cdot (\delta_{\lambda} \otimes v )\big(\ell\inv x, y\big) \\
  &= u\big(\kappa\inv x y \big) v(y) \delta_{\lambda}\big(\kappa\inv x\big) = u\big(\kappa\inv x y \big) v(y) \delta_{\kappa \lambda}(x) \\
  &= u(\lambda y ) v(y) \delta_{\kappa \lambda}(x)
  = \delta_{\kappa \lambda} \otimes (( L_{\lambda}(u)) \cdot v).
\end{align*}
The second line follows from the fact that the only non-zero term of the sum occurs when $ \ell = \kappa $ and the second to last equality follows from the fact that the input for $ u $ only matters when $ x = \kappa \lambda $.
\end{proof}

\subsection{Compatibility for matrix coefficients}
We will now prove a version of Propositions~\ref{proposition:convolution.of.non.irreducible.coefficients} and~\ref{proposition:new.weak.compatibility} that is generalized to the case of a~compact singular Lie group.
\begin{Proposition}\label{proposition:new.weak.compat}
   Let $ \pi $ be a finite-dimensional representation of $ G $ and let $ \kappa_1, \kappa_2,$ $\lambda_1, \lambda_2 \in K $. Given functions $ u, v \in C^\infty(G) $ and $ 1 \le i,k \le d_\pi $, we have that
  \[
   \sum_{j = 1}^{d_\pi} \big( u^{\kappa_1} \starcirc \pi_{ij}^{\lambda_1} \big) \stardot \big( v^{\kappa_2} \starcirc \pi_{jk}^{\lambda_2} \big) =  ( (L_{\lambda_1}u * L_{\lambda_2} v ) \cdot \pi_{ik} )^{\kappa_1 \lambda_1 \kappa_2 \lambda_2}.
  \]
\end{Proposition}
\begin{proof}
We do a direct calculation, taking advantage of Lemmas~\ref{lemma:compute.star.circ} and~\ref{lemma:translation.and.convolution}, and Proposition~\ref{proposition:convolution.of.non.irreducible.coefficients},
\begin{align*}
  \sum_{j = 1}^{d_\pi} \big( u^{\kappa_1} \starcirc \pi_{ij}^{\lambda_1} \big) \stardot \big( v^{\kappa_2} \starcirc \pi_{jk}^{\lambda_2} \big)
  &= \sum_{j = 1}^{d_\pi} ( L_{\lambda_1} u \cdot \pi_{ij})^{\kappa_1\lambda_1} \stardot ( L_{\lambda_2}
 v \cdot \pi_{jk} )^{\kappa_2 \lambda_2} \\
  &=\sum_{j=1}^{d_\pi} \left( (L_{\lambda_1} u \cdot \pi_{ij}) * (L_{\lambda_2} v \cdot \pi_{jk})\right)^{\kappa_1 \lambda_1 \kappa_2 \lambda_2} \\
  &= \left( (L_{\lambda_1} u * L_{\lambda_2} v ) \cdot \pi_{ik} \right)^{\kappa_1 \lambda_1 \kappa_2 \lambda_2 }.\tag*{\qed}
  \end{align*}\renewcommand{\qed}{}
\end{proof}

Similar to the case for compact Lie groups, we state the main theorem, which gives us a~``weak compatibility'' law.
\begin{Theorem}\label{theorem:new.main}
 Let $ \pi $ be a finite-dimensional, unitary representation of $ G $. Let $ u, v \in C^\infty(G) $ be smooth functions and $ \kappa_1, \kappa_2, \lambda_1, \lambda_2 \in K $. Then for
 $1 \le i,k \le d_{\pi} $, we have that
  \[
   \sum_{j = 1}^{d_\pi} \big( u^{\kappa_1} \starcirc \pi_{ij}^{\lambda_1} \big) \stardot \big( v^{\kappa_2} \starcirc \pi_{jk}^{\lambda_2} \big) =   \sum_{j = 1}^{d_\pi} \big( u^{\kappa_1} \stardot  v^{\kappa_2} \big) \starcirc \big(\pi_{ij}^{\lambda_1} \stardot \pi_{jk}^{\lambda_2} \big).
  \]
\end{Theorem}

\begin{proof}
By combining Proposition~\ref{proposition:new.weak.compat} with Corollary~\ref{corollary:conv.of.non.irred}, the left-hand side of the equation can be seen to be equal to
\smash{$\sum_{j=1}^{d_\pi} \left( (L_{\lambda_1} u * L_{\lambda_2} v)  \cdot (\pi_{ij} * \pi_{jk}) \right)^{\kappa_1 \lambda_1 \kappa_2 \lambda_2 }$}.
 On the other hand, if we compute the right-hand side, we get
 \begin{align*}
     \sum_{j = 1}^{d_\pi} \big( u^{\kappa_1} \stardot  v^{\kappa_2} \big) \starcirc \big(\pi_{ij}^{\lambda_1} \stardot \pi_{jk}^{\lambda_2} \big)
     &=   \sum_{j = 1}^{d_\pi} (u * v)^{\kappa_1 \kappa_2} \starcirc (\pi_{ij} *\pi_{jk})^{\lambda_1 \lambda_2} \\
     &= \sum_{j=1}^{d_\pi} \big( L_{\lambda_1 \lambda_2} (u * v) \cdot (\pi_{ij} * \pi_{jk} ) \big)^{\kappa_1 \kappa_2 \lambda_1 \lambda_2 }.
 \end{align*}
 Therefore, the desired equality holds if
$\kappa_1 \kappa_2 \lambda_1 \lambda_2 = \kappa_1 \lambda_1 \kappa_2 \lambda_2$
 and
$ L_{\lambda_1 \lambda_2} (u * v) = L_{\lambda_1} u * L_{\lambda_2} v$.
Since $ K $ is assumed to be a subgroup of the center, it is abelian so the first equality holds. The second equality follows from Lemma~\ref{lemma:translation.and.convolution}.
\end{proof}

\subsection{Further remarks on the singular Lie group case}
Similar to our observations in Section~\ref{subsection:further.remarks}, there are a few similar formulas one can obtain by looking at special cases of Theorem~\ref{theorem:new.main}.

If we additionally assume that $ u = \sigma_{ab} $ and $ v = \sigma_{bc} $ for some representation $ \sigma $ of $ G $. Theorem~\ref{theorem:new.main} instead states that
  \[
   \sum_{j = 1}^{d_\pi} \big( \sigma_{ab}^{\kappa_1} \starcirc \pi_{ij}^{\lambda_1} \big) \stardot \big( \sigma_{bc}^{\kappa_2} \starcirc \pi_{jk}^{\lambda_2} \big) =   \sum_{j = 1}^{d_\pi} \big( \sigma_{ab}^{\kappa_1} \stardot  \sigma_{bc}^{\kappa_2} \big) \starcirc \big(\pi_{ij}^{\lambda_1} \stardot \pi_{jk}^{\lambda_2} \big).
  \]
  If $ \sigma $, $ \pi $ and the tensor product representation $ \sigma \otimes \pi $ are all irreducible, then it is possible to show that
    \smash{$
    \big( \sigma_{ab}^{\kappa_1} \starcirc \pi_{ij}^{\lambda_1} \big) \stardot \big( \sigma_{bc}^{\kappa_2} \starcirc \pi_{jk}^{\lambda_2} \big) =  \big( \sigma_{ab}^{\kappa_1} \stardot  \sigma_{bc}^{\kappa_2} \big) \starcirc \big(\pi_{ij}^{\lambda_1} \stardot \pi_{jk}^{\lambda_2} \big)$}.
  In fact, this is precisely what occurs in the case of a noncommutative torus. Compare the above formula to the computation from Section~\ref{Section:noncommutative.torus}.

   Theorem~\ref{theorem:new.main} can be thought of as ``equivalent'' to compatibility of the groupoid operations (i.e., equation~\eqref{eqn:intertwininglaw}) in the following sense. Even if $ K $ is not a subgroup of the center, the operations $ \circ $ and $ \bullet $ can still be defined. However, they will not be compatible. Theorem~\ref{theorem:new.main} holds for such operations if and only if $ K $ is a subgroup of the center. The crucial step occurs in the final lines of the proof where one needs that $ L_\kappa (u * v) = u * L_{\kappa} v $ for arbitrary functions~$ u $ and $ v $ and elements $ \kappa \in K $. It is not too difficult to convince oneself that this can only occur when $ K \subset Z(G) $.

\section{Conclusions and towards a more general approach}
The computations we performed in the last three sections are rather specialized to their specific cases. Furthermore, the results obtained in Sections~\ref{Section:compact.Lie.groups} and~\ref{Section:compact.singular.Lie.groups} are a bit unsatisfying as they involve compromises in the form of the compatibility law that seem somewhat artificial.

Our methods so far have focused on looking purely at the algebra structures on $ C(G) $ while only relying on the structure of the double groupoid to perform the computations. However, double groupoids admit a rich variety of structures and it is certain that more can be said if we permit ourselves to ``remember'' more of this data. In particular, associated to any double groupoid is an infinite family of groupoids formed out of spaces of composable elements. Each of these groupoids comes equipped with their own convolution operations and more.

To see what we mean by this, consider first a general double groupoid
 \begin{equation*}
  \begin{tikzcd}
\CG \darrow \rarrow & \CK \darrow \\
\CH \rarrow & \,M.
\end{tikzcd}
\end{equation*}
Write $ \CG^{\boxplus} $ to denote the set of ``composable $2\times 2$ squares'' (for a precise definition, see equation~\eqref{equation:groupoid.box}). The compatibility law for the double groupoid says that the following diagram commutes:
\[
\begin{tikzcd}
  \CG^\boxplus \arrow[r] \arrow[d] & \CG \times_{s^V,t^V} \CG \arrow[d] \\
  \CG \times_{s^H,t^H} \CG \arrow[r] &\,  \CG,
\end{tikzcd}
\]
where the horizontal maps correspond to multiplying horizontally and the vertical maps correspond to multiplying vertically.

Using a double Haar system, one can obtain measures on the fibers of these maps. Since convolution is essentially given by integration, along a multiplication map one obtains a commutative diagram
\[
\begin{tikzcd}
  C^\infty_c(\CG^\boxplus) \arrow[d] \arrow[r] & C_c^\infty(\CG \times_{s^V, t^V} \CG ) \arrow[d] \\
  C^\infty_c(\CG \times_{s^H, t^H} \CG) \arrow[r] &\, C^\infty_c(\CG).
\end{tikzcd}
\]
Some work is required to prove this carefully and it relies on the invariance properties for double Haar systems.

It is our intention to explore this direction further in a future article, where we take a closer look at double Haar systems and the double simplicial complex of a double groupoid.

\subsection*{Acknowledgements}

The authors would like to thank Xiang Tang for some helpful comments on the topic. The authors would like to thank the participants of the Weekend Workshop on Representation Theory and Noncommutative Geometry in Washington University in St.~Louis for their valuable feedback and suggestions. The authors would like to also acknowledge the anonymous referee who provided several ideas for improvements to this article. This article is based on work that was supported by the National Science Foundation (Award Numbers 2137999 and 2213097).

\pdfbookmark[1]{References}{ref}
\LastPageEnding

\end{document}